\documentclass{amsart}
\usepackage{amsmath, amssymb}
\usepackage[all]{xy}

\usepackage{hyperref}%[colorlinks=true]

\theoremstyle{plain}
\newtheorem*{thm*}{Theorem}
\newtheorem*{prop*}{Proposition}
\newtheorem*{rem*}{Remark}

\newtheorem{thm}{Theorem}[section]
\newtheorem{cor}[thm]{Corollary}
\newtheorem{defi}[thm]{Definition}
\newtheorem{prop}[thm]{Proposition}
\newtheorem{lm}[thm]{Lemma}

\newtheorem{claim*}{Claim}

\newtheorem{rem}[thm]{Remark}
\newtheorem{nota}[thm]{Notation}
\numberwithin{equation}{thm}

\newcommand{\A}{{\mathcal{A}}}
\newcommand{\C}{{\mathcal{C}}}
\newcommand{\D}{{\mathcal{D}}}
\newcommand{\F}{{\mathcal{F}}}
\newcommand{\E}{{\mathcal{E}}}
\newcommand{\FF}{{\mathbb{F}_2}}

\newcommand{\Eq}{{\ensuremath{\mathcal{E}_{q}}}}
\newcommand{\Eqd}{{\mathcal{E}_{q}^{\mathrm{deg}}}}
\newcommand{\Tq}{{\ensuremath{\mathcal{T}_{q}}}}

\newcommand{\Fq}{{\ensuremath{\mathcal{F}_{quad}}}}
\newcommand{\Gq}{{\ensuremath{\mathcal{F}_{iso}}}}

\newcommand{\M}[1]{{\ensuremath{\mathrm{Mix}_{#1}}}}

\newcommand{\m}[1]{{\ensuremath{\mathrm{\Sigma}_{#1}}}}

\begin{document}

\title[Generic representations of orthogonal groups:  projective functors]{Generic representations of orthogonal groups:  projective functors in the category $\mathcal{F}_{quad}$}

\author{Christine Vespa}

\address{ Ecole Polytechnique F\'ed\'erale de Lausanne, Institut de G\'eom\'etrie, Alg\`ebre et Topologie, Lausanne, Switzerland. }                         
\email{christine.vespa@epfl.ch}    

\date{\today}

\begin{abstract}
In this paper, we continue the study of the category of functors $\Fq$, associated to $\FF$-vector spaces equipped with a nondegenerate quadratic form, initiated in
\cite{math.AT/0606484} and \cite{Vespa2}. We define a filtration of the standard
projective objects in $\Fq$; this refines to give a decomposition into
indecomposable factors of the two first standard projective objects in
$\Fq$: $P_{H_0}$ and $P_{H_1}$. As an application of these two decompositions, we give a complete
description of the polynomial functors of the category $\Fq$.\\
$\  $\\
\textit{Mathematics Subject Classification:} 18A25, 16D90, 20C20.
$\ $\\
\textit{Keywords}: functor categories; quadratic forms over $\FF$;
Mackey functors; representations of orthogonal groups over $\FF$.
\end{abstract}

\maketitle

%\tableofcontents

\section*{Introduction}
In the paper \cite{math.AT/0606484} we defined the category of functors $\Fq$
from a category having as objects the nondegenerate $\FF$-quadratic
spaces  to the category $\E$ of $\mathbb{F}_2$-vector spaces, where $\FF$ is
the field with two elements. The motivation for the construction of
this category is to obtain an analogous framework for the orthogonal
groups over $\FF$, to that which exists for the general linear groups. We
recall that the category $\F$ of
functors from the category $\E^f$ of finite dimensional
$\mathbb{F}_2$-vector spaces to the category $\E$ of all $\mathbb{F}_2$-vector spaces is a very useful tool for the study of the stable cohomology
of the general linear groups with suitable coefficients (see
\cite{FFSS}). Another motivation, in topology, for the study of the category $\F$ is the connection which exists between this category and unstable modules over the Steenrod algebra (see \cite{Sch}). In order to have a good understanding of the category
$\Fq$, we seek to classify its simple objects.
We constructed in \cite{math.AT/0606484} two families of simple objects in
$\Fq$. The first one is obtained by the fully-faithful, exact functor $\iota: \F \rightarrow
\Fq$, defined in \cite{math.AT/0606484}, which preserves simple objects. By \cite{KuhnII}, the simple
objects in $\F$ are in one-to-one correspondence with the irreducible
representations of finite general linear groups over $\FF$. The second family is obtained
by the fully-faithful, exact functor $\kappa: \Gq \rightarrow
\Fq$, which preserves simple objects, where $\Gq$ is
equivalent to the product of the categories of modules over the
orthogonal groups of possibly degenerate quadratic forms. In \cite{Vespa2}, we constructed two families of
simple objects in the category $\Fq$ which are neither in the image of
$\iota$ nor in the image of $\kappa$. These simple objects are
subfunctors of the tensor product between an object in the image of
$\iota$ and an object in the image of $\kappa$. We proved that
these simple objects in $\Fq$ are the composition factors of two
particular mixed functors, defined in \cite{Vespa2}.

The aim of this paper is to begin a programme to obtain a complete
classification of the simple objects in $\Fq$. Accordingly, we seek
to decompose the projective generators of this category into 
indecomposable factors 
and to obtain the simple factors of these indecomposable factors. 
This paper begins the study of the standard projective
objects in the category $\Fq$. Although explicit decompositions of
all the projective generators are not provided in this paper, we give several useful tools, results
and examples for the realization of this programme. Furthermore, we deduce from the results contained in this paper several interesting consequences for the structure of the category $\Fq$.
In work in progress, we obtain a general decomposition of standard projective object $P_H$ of $\Fq$ which is indexed by the subspaces of $H$. Here we present explicit decompositions of the standard projective objects  associated to ``small'' quadratic spaces, since these decompositions play a fundamental r\^{o}le in the category $\Fq$ (for example, for the description of the polynomial functors of $\Fq$). Furthermore, recall that the decompositions of the injective standard $I_{\FF}^{\F}$ of the category $\F$ and thus, by duality, that of the projective standard $P_{\FF}^{\F}$, is fundamental for the comprehension of the other injective standards of $\F$. Hence, the decompositions of the two smaller projective standard of $\Fq$ represent an important step in the understanding of the category $\Fq$.

We briefly summarize the contents of this paper. After some recollections on the category $\Fq$, where we recall the definitions of the isotropic functors and the mixed functors, we define a filtration
of the standard projective objects $P_V$ in $\Fq$: 
$$0 \subset P_V^{(0)} \subset P_V^{(1)} \subset \ldots \subset
P_V^{(\mathrm{dim}(V)-1)} \subset P_V^{(\mathrm{dim}(V))}=P_V.$$
We obtain a general
description of the two extremities of this filtration. 

\begin{thm*}
Let $V$ be a nondegenerate $\FF$-quadratic space.
\begin{enumerate}
\item \label{PV0-1}
There is a natural equivalence: $P_V^{(0)} \simeq \iota(P^{\F}_{\epsilon(V)})$,
where $\iota : \F \rightarrow \Fq$, $\epsilon$ is the functor that forgets the quadratic form and $P^{\F}_{\epsilon(V)}$ is the standard projective object in $\F$ associated to the vector space $\epsilon(V)$.
\item \label{PV0-2}
The functor $P_V^{(0)}$ is a direct summand of $P_V$.
\end{enumerate}
\end{thm*}

\begin{prop*} 
Let $V$ be a nondegenerate $\FF$-quadratic space, we have a natural equivalence:
$$P_V / P_V^{(\mathrm{dim}(V)-1)} \simeq \kappa(\mathrm{iso}_V)$$
where $\kappa: \Gq \rightarrow \Fq$ and $\mathrm{iso}_V$ is an isotropic functor in $\Gq$.
\end{prop*}

An important consequence of the Theorem concerning  the functor $P_V^{(0)}$ is given in following result.

\begin{thm*} 
The category $\iota(\F)$ is a thick subcategory of  $\Fq$.
\end{thm*}

Then by an explicit study of the filtration of the functors $P_{H_0}$ and $P_{H_1}$  we obtain the following fundamental decompositions of these two standard projective functors.

\begin{thm*} 
\begin{enumerate}
\item
The standard projective object $P_{H_0}$ admits the following decomposition:
$$P_{H_0}=\iota(P^{\F}_{\FF^{\oplus 2}}) \oplus
(\M{0,1}^{\oplus 2} \oplus \M{1,1}) \oplus \kappa(iso_{H_0})$$
where $\M{0,1}$ and $\M{1,1}$ are two mixed functors and
$iso_{H_0}$ is an isotropic functor.
\item
The standard projective object $P_{H_1}$ admits the following decomposition:
$$P_{H_1}=\iota(P^{\F}_{\FF^{\oplus 2}}) \oplus
\M{1,1}^{\oplus 3}  \oplus \kappa(iso_{H_1})$$
where $\M{1,1}$ is a mixed functor and
$iso_{H_1}$ is an isotropic functor.
\end{enumerate}
\end{thm*}

These decompositions have several interesting consequences. Firstly, thanks to this theorem we can complete the study of the functors  \M{0,1} and \M{1,1} started in \cite{Vespa2} by the following result.

\begin{prop*}
The functors \M{0,1} and \M{1,1} are indecomposable.
\end{prop*}

We want to emphasize that the complete structure of the direct summands of the decompositions of $P_{H_0}$ and $P_{H_1}$ is understood. The structure of the isotropic functors is given in \cite{math.AT/0606484}, those of the mixed functors \M{0,1} and \M{1,1}  is the main result of \cite{Vespa2} and is completed by the previous proposition and those of $P^{\F}_{\FF^{\oplus 2}}$ follows from \cite{KuhnII}.
Then, these decompositions give rise to a classification of the simple
functors $S$ of $\Fq$ such that $S(H_0) \ne 0$ or $S(H_1) \ne 0$. 

\begin{prop*}
The isomorphism classes of non-constant simple functors of $\Fq$ such
that either  $S(H_0) \ne 0$ or  $S(H_1) \ne 0$ are:
$$\iota(\Lambda^1),\ \iota(\Lambda^2),\ \iota(S_{(2,1)}),\
\kappa(iso_{(x,0)}),\ \kappa(iso_{(x,1)}),\ R_{H_0},\ R_{H_1},\ S_{H_1}$$
where $ R_{H_0},\ R_{H_1}$ and $ S_{H_1}$ are the simple functors
introduced in Corollary \ref{1.7}.
\end{prop*}

These decompositions also allow us to derive some homological calculations in the category $\Fq$.

\begin{prop*}
For $n$ a natural number, we have:
$$\mathrm{Ext}^n_{\Fq}(R_{H_0},R_{H_0}) \simeq \FF \mathrm{\quad and
  \quad} \mathrm{Ext}^n_{\Fq}(R_{H_1},R_{H_1}) \simeq \FF $$
where $R_{H_0}$ and $R_{H_1}$ are the simple functors introduced in
Corollary \ref{1.7}.
\end{prop*}

Finally, after having introduced the notion of polynomial functor for the
category $\Fq$, which generalizes that for $\F$, we obtain the following result as an application of the classification of the simple
functors $S$ of $\Fq$ such that $S(H_0) \ne 0$ or $S(H_1) \ne 0$ and of the thickness of the subcategory $\iota(\F)$ of $\Fq$.

\begin{thm*}
The polynomial functors of $\Fq$ are in the image of the
functor $\iota: \F \rightarrow \Fq$.
\end{thm*}

Most of the results of this paper are contained in the Ph.D. thesis of the author  \cite{Vespa-these}.

%%%%%%%%%%%%%%%%%%%%%%%%%%%%%%%%%%%%%%%%%%%%%%
\section{The category $\Fq$: some recollections} \label{1}
We recall in this section some definitions and results about the
category $\Fq$ obtained in \cite{math.AT/0606484}.

Let $\Eq$ be the category having as objects finite dimensional $\FF$-vector spaces equipped with a non
degenerate quadratic form and with morphisms linear maps that
preserve the quadratic forms. By the classification of quadratic forms
over the field $\FF$ (see, for instance, \cite{Pfister}) we know that only spaces of even
dimension can be nondegenerate and, for a fixed even dimension, there are  two non-equivalent nondegenerate spaces, which are distinguished by
the Arf invariant. We will denote by $H_0$ (resp. $H_1$) the
nondegenerate quadratic space of dimension two such that $\mathrm{Arf}(H_0)=0$
(resp. $\mathrm{Arf}(H_1)=1$). The orthogonal sum of two nondegenerate quadratic spaces $(V,q_V)$ and $(W,q_W)$ is, by definition, the quadratic space $(V \oplus W,q_{V \oplus W})$ where $q_{V \oplus W}(v,w)=q_V(v)+q_W(w)$. Recall that the spaces $H_0 \bot H_0$
and $H_1 \bot H_1$ are isomorphic. Observe that the morphisms of $\Eq$
are injective linear maps and this category does not admit push-outs
or pullbacks. There exists a pseudo push-out in $\Eq$ that allows us to generalize the construction of the category of
co-spans of B\'{e}nabou  \cite{Benabou} and thus to define the category $\Tq$ in which there
exist retractions.
\begin{defi} \label{1.1}
The category $\Tq$ is the category having as objects those of $\Eq$
and, for $V$ and $W$ objects in $\Tq$, $\mathrm{Hom}_{\Tq}(V,W)$ is
the set of equivalence classes of diagrams in $\Eq$ of the form $ V
\xrightarrow{f} X \xleftarrow{g} W $  for the equivalence relation
generated by the relation $\mathcal{R}$ defined as follows: $ V \xrightarrow{f} X_1 \xleftarrow{g} W \quad  \mathcal{R}\quad   V
\xrightarrow{u} X_2 \xleftarrow{v} W $ if there exists a morphism
$\alpha$ of $\Eq$ such that $\alpha \circ f=u$ and $\alpha \circ g=v$.
The composition is defined using the pseudo push-out. The morphism of
$\mathrm{Hom}_{\Tq}(V,W)$ represented by the diagram $ V
\xrightarrow{f} X \xleftarrow{g} W $ will be denoted by $[ V
\xrightarrow{f} X \xleftarrow{g} W ]$.
\end{defi}

\begin{rem} \label{1.2}
A morphism of $\mathrm{Hom}_{\Tq}(V,W)$ is represented by a diagram of the form: $V \xrightarrow{g} W \bot W' \xleftarrow{i_W}
W$, where $i_W$ is the canonical inclusion. In the following, we will
use this representation of a morphism, without further comment.
\end{rem}

 By definition,
the category $\Fq$ is the category of functors from $\Tq$ to $\E$.
Hence $\Fq$ is
abelian and has enough projective objects. By the Yoneda lemma, for any object $V$ of $\Tq$, the functor $P_V=\FF \lbrack
\mathrm{Hom}_{\Tq}(V,-) \rbrack $ is a projective object and there is
a natural isomorphism: $\mathrm{Hom}_{\Fq}(P_V, F) \simeq F(V)$,
for all objects $F$ of $\Fq$. The set of functors $\{ P_V | V \in
\mathcal{S} \}$, named the standard projective objects in $\Fq$, is a set of
projective generators of $\Fq$, where $\mathcal{S}$ is a set of
representatives of isometry classes of nondegenerate quadratic
spaces.

There is a forgetful functor $\epsilon: \Tq \rightarrow \E^f$ in $\Fq$, defined by $\epsilon(V)=\mathcal{O}(V)$ and 
$$\epsilon([V \xrightarrow{f} W \bot W' \xleftarrow{g} W ])=p_g\circ
\mathcal{O}(f)$$
where $p_g$ is the orthogonal projection from $W \bot W'$ to $W$ and $\mathcal{O}: \Eq \rightarrow \E^f$ is the functor which forgets the quadratic form.  By the fullness of the functor $\epsilon$ and an argument of essential surjectivity, we obtain the following theorem.
\begin{thm} \label{1.3} \cite{math.AT/0606484}
There is a functor $\iota: \F \rightarrow \Fq$, which is exact,
fully faithful and preserves simple objects.
\end{thm}

In order to define another subcategory of $\Fq$, we consider the
category $\Eqd$ having as objects finite dimensional $\FF$-vector spaces
equipped with a (possibly degenerate) quadratic
form and with morphisms injective linear maps which
preserve the quadratic forms. The category $\Eqd$ admits
pullbacks; consequently the category of spans $\mathrm{Sp}(\Eqd)$ (\cite{Benabou}) is
defined. By definition, the category $\Gq$ is the category of functors
from $\mathrm{Sp}(\Eqd)$ to $\E$. As in the case of the category $\Fq$, the
category $\Gq$ is abelian and has enough projective objects: by the Yoneda lemma, for any object $V$ of $\mathrm{Sp}(\Eqd)$, the functor $Q_V=\FF \lbrack
\mathrm{Hom}_{\mathrm{Sp}(\Eqd)}(V,-) \rbrack $ is a projective object
in $\Gq$. We define a particular family of functors of $\Gq$, the isotropic
functors, which form a set of projective generators and injective
cogenerators of $\Gq$.
The category $\Gq$ is related to $\Fq$ by the following theorem.
\begin{thm} \label{1.4} \cite{math.AT/0606484}
There is a functor $\kappa: \Gq \rightarrow \Fq$, which is exact,
fully-faithful and preserves simple objects.
\end{thm} 
We obtain the classification of the simple objects of the category
$\Gq$ from the following theorem.
\begin{thm} \label{1.5}  \cite{math.AT/0606484}
There is a natural equivalence of categories 
$$\Gq \simeq \prod_{V \in \mathcal{S}} \FF[O(V)]-mod$$
where $\mathcal{S}$ is a set of representatives of isometry classes of
quadratic spaces (possibly degenerate) and $O(V)$ is the orthogonal group.
\end{thm}
The object of $\Gq$ which corresponds, by this equivalence, to the
module $\FF[O(V)]$ is the isotropic functor $iso_V$, defined in
\cite{math.AT/0606484}. Recall that, as a vector space, $iso_V(W)$ is isomorphic to the subspace of $Q_V(W)$ generated
by the elements $[ V \xleftarrow{\mathrm{Id}} V \rightarrow W]$. 
%The canonical generator $[ V \xleftarrow{\mathrm{Id}} V \xrightarrow{h} W]$ of $iso_V(W)$  will be denoted by $[ V \xrightarrow{h} W]$ or, more simply, by $[h]$.

A straightforward consequence of the classification of simple objects of $\Gq$ given in
Theorem \ref{1.5} is given in the following corollary. Recall that, by definition, an object $F$ of $\Fq$ is finite if it has
a finite composition series with simple subquotients.
\begin{cor} \label{1.6}
The isotropic functors $\kappa(iso_V)$ are finite in the category $\Fq$.
\end{cor}

In section \ref{3}, we will require the composition series for
the isotropic functors associated to some small quadratic spaces. 
For $\alpha \in \{ 0,1 \}$, let $(x,\alpha)$ be the degenerate quadratic space of
dimension one generated by $x$ such that $q(x)=\alpha$. Since the
orthogonal groups $O(x,0)$ and $O(x,1)$ are trivial and $O(H_0) \simeq
\mathfrak{S}_2$ and $O(H_1) \simeq \mathfrak{S}_3$, we deduce from Theorem \ref{1.5} and
\ref{1.4}, the following corollary.
\begin{cor} \label{1.7} 
\begin{enumerate}
\item
The functors $\kappa(iso_{(x,0)})$ and $\kappa(iso_{(x,1)})$
are simple in $\Fq$.
\item
The functor $\kappa(iso_{H_0})$ is indecomposable. We have the following
non-split short exact sequence:
$$0 \rightarrow R_{H_0} \rightarrow \kappa(iso_{H_0}) \rightarrow
R_{H_0} \rightarrow 0$$
where $R_{H_0}$ is the functor obtained from the trivial
representation of $O(H_0)$.
\item
The functor $\kappa(iso_{H_1})$ admits the following decomposition:
$$\kappa(iso_{H_1})=F_{H_1} \oplus (S_{H_1})^{\oplus 2} $$
where $S_{H_1}$ is the functor obtained from the natural representation of $O(H_1)$ and $F_{H_1}$ is an indecomposable functor for which we
have the following non-split short exact sequence:
$$0 \rightarrow R_{H_1} \rightarrow F_{H_1} \rightarrow
R_{H_1} \rightarrow 0$$
where $R_{H_1}$ is the functor obtained from the trivial
representation of $O(H_1)$.
\end{enumerate}
\end{cor}
In \cite{Vespa2}, we define a new family of functors of $\Fq$, named
the mixed functors and we decompose two particular functors of this
family: the functors \M{0,1}\ and \M{1,1}. We recall the following
description of these functors.
\begin{prop} \cite{Vespa2}
For $\alpha \in \{0,1 \}$, the functors $\M{\alpha,1}: \Tq
\rightarrow \E$ are defined by
$$\M{\alpha,1}(V)=\FF[S_V]$$
where $S_V=\{(v_1,v_2)\ | v_1 \in V, v_2 \in V,\ q(v_1+v_2)=
\alpha,\ B(v_1,v_2)=1\}$ and 
$$\M{\alpha,1}([V \xrightarrow{f} W \bot L \xleftarrow{i_W} W])[
(v_1,v_2)]=
\left\lbrace
\begin{array}{cc}
[(p_W \circ f(v_1), p_W \circ f (v_2))] &\mathrm{if\ } f(v_1+v_2) \in W\\
0 & \mathrm{otherwise}
\end{array}
\right.$$
where $p_W$ is the orthogonal projection.
\end{prop}
In \cite{Vespa2}, for a positive integer $n$, we defined subfunctors $L^n_{\alpha}$ of $\iota(\Lambda^n) \otimes
\kappa(iso_{(x,\alpha)})$, where $\Lambda^n$ is the $n$th exterior
power and we proved that these functors are simple. The functor $L^1_{\alpha}$ is equivalent to the functor
$\kappa(iso_{(x,\alpha)})$. We obtain the following result.

\begin{thm} \label{1.9} \cite{Vespa2}
Let $\alpha$ be an element in $\{0,1 \}$.
\begin{enumerate}
\item
The functor $\M{\alpha,1}$ is infinite.
\item
There exists a subfunctor $\m{\alpha,1}$ of $\M{\alpha,1}$ such that we have the following  short exact sequence
$$0 \rightarrow \m{\alpha,1} \rightarrow \M{\alpha,1} \rightarrow \m{\alpha,1} \rightarrow 0.$$
\item
The functor $\m{\alpha,1}$ is uniserial with unique composition series given by the decreasing filtration given by the subfunctors $k_d \m{\alpha,1}$ of $\m{\alpha,1}$:
$$\ldots \subset k_d \m{\alpha,1} \subset \ldots \subset k_1
\m{\alpha,1} \subset k_0 \m{\alpha,1}= \m{\alpha,1}.$$
\begin{enumerate}
\item
The head of $\m{\alpha,1}$ (i.e. $\m{\alpha,1} / k_1
\m{\alpha,1} $) is isomorphic to the functor  $\kappa(iso_{(x,\alpha)})$ where $iso_{(x,\alpha)}$ is a simple object in $\Gq$.
\item
For $d>0$ 
$$k_d \m{\alpha,1}/k_{d+1} \m{\alpha,1} \simeq L^{d+1}_\alpha $$
where $L^{d+1}_\alpha$ is a simple object of the category $\Fq$ that is neither in the image of
$\iota$ nor in the image of $\kappa$. 

The functor $L^{d+1}_\alpha$ is a subfunctor of $\iota(\Lambda^{d+1}) \otimes \kappa(iso_{(x,\alpha)})$, where $\Lambda^{d+1}$ is the $(d+1)$st exterior power functor.
\end{enumerate}

\end{enumerate}
\end{thm}

\section{Filtration of the standard projective functors $P_V$ of
  $\Fq$} \label{2}

In this section, we define a filtration of the standard projective
functors $P_V$ of $\Fq$. This construction gives rise to an essential
tool to obtain, in section \ref{3}, the direct decompositions of the projective objects $P_{H_0}$ and $P_{H_1}$
of $\Fq$, into indecomposable summands. 

After defining this filtration, we will deduce general
results about the projective $P_V$ of $\Fq$. In Theorem
\ref{2.6} we prove that the rank zero part is a direct summand of
$P_V$ and we identify this functor. This result allows us to prove that
$\iota(\F)$ is a thick subcategory of $\Fq$. We will also show that
the top quotient of this filtration is isomorphic to $\kappa(iso_V)$, where $iso_V$ is the isotropic functor.

\subsection{Definition of the filtration}
We recall that a morphism in $\Tq$ from $V$ to $W$, where $V$ and $W$
are nondegenerate quadratic spaces, is represented by a diagram 
$V \rightarrow X \leftarrow W.$

\begin{defi} \label{2.1}
A morphism $[V \rightarrow X \leftarrow W]$ in $\Tq$ has rank equal to $i$ if
the pullback in $\Eqd$ of the diagram $V \rightarrow X \leftarrow W$
is a quadratic space of dimension $i$.
\end{defi}

\begin{nota} 
We denote by $\mathrm{Hom}^{(i)}_{\Tq}(V,W)$ the subset of $\mathrm{Hom}_{\Tq}(V,W)$ of morphisms
 of rank less than or equal to $i$.
\end{nota}
We have the following proposition:
\begin{prop} \label{2.3}
For $W$ an object in $\Tq$, the following subvector space of $P_V(W)$:
$$P_V^{(i)}(W)=\FF[\mathrm{Hom}^{(i)}_{\Tq}(V,W)]$$
defines a subfunctor $P_V^{(i)}$ of $P_V$.
\end{prop}
\begin{proof}
It is sufficient to verify that for all morphisms $f=[W \rightarrow Y
\leftarrow Z]$ of $\Tq$ and $g=[V \rightarrow X
\leftarrow W]$ of $\mathrm{Hom}^{(i)}_{\Tq}(V,W)$, the composition  $f
\circ g$ has rank less than or equal to $i$. The composition $f \circ g$ is represented by the following commutative diagram:

$$\xymatrix{
P' \ar[rr] \ar[d] & & Z \ar[d]\\
P \ar[r] \ar[d] & W \ar[r] \ar[d] & Y \ar[d] \\
V \ar[r] & X \ar[r] & X \underset{W}{\bot} Y  
}$$
where $P$ and $P'$ are the pullbacks and $X \underset{W}{\bot} Y$
is the pseudo push-out defined in \cite{math.AT/0606484}. Consequently: 
$f \circ g=[V \rightarrow X
\underset{W}{\bot} Y \leftarrow Z].$
Since $[V \rightarrow X
\leftarrow W]$ is an element of $P_V^{(i)}(W)$, we know that the
dimension of $P$ is less than or equal to $i$. We deduce from the
injectivity of the morphisms of
$\Eqd$, that $P'$ has dimension smaller than or equal to $i$.
\end{proof}
The following lemma is a straightforward consequence of Definition \ref{2.1}.
\begin{lm}
There exists a natural equivalence: $P_V^{(\mathrm{dim}(V))}\simeq P_V.$
\end{lm}
We deduce the following proposition.
\begin{prop}
The functors $P_V^{(i)}$, for $i=0, \ldots, \mathrm{dim}(V)$,
define an increasing filtration of the functor $P_V$.
\end{prop}
\begin{proof}
The inclusion of vector spaces $P_V^{(i)}(W) \subset P_V^{(i+1)}(W)$
is clear, for $W$ an object in $\Tq$. Consequently, $P_{V}^{(i)}$ is a
subfunctor of $P_{V}^{(i+1)}$ by the proposition \ref{2.3}.
\end{proof}

\subsection{Extremities of the filtration}

In the previous section, we have obtained, for all objects $V$ of
$\Tq$, the following filtration of the functor $P_V$:
$$0 \subset P_V^{(0)} \subset P_V^{(1)} \subset \ldots \subset
P_V^{(\mathrm{dim}(V)-1)} \subset P_V^{(\mathrm{dim}(V))}=P_V.$$
The aim of this section is to study the two extremities of this
filtration, namely, the functor $P_V^{(0)}$ and the quotient
$P_V/P_V^{(\mathrm{dim}(V)-1)}$.

\subsubsection{The functor $P_V^{(0)}$} \label{PV0}
For $V$ an object in $\Tq$, we recall that the functor  $P^{\F}_{\epsilon(V)}$ of $\F$ defined by  $P^{\F}_{\epsilon(V)}(-)=\FF
[\mathrm{Hom}_{\E^f}(\epsilon(V),-)]$, where $\epsilon : \Tq
\rightarrow \E^f$  is the forgetful functor of $\Fq$, is  projective, by the Yoneda lemma.
The aim of this paragraph is to prove the following theorem:
\begin{thm} \label{2.6}
Let $V$ be an object of $\Tq$.
\begin{enumerate}
\item \label{PV0-1}
There is a natural equivalence: $P_V^{(0)} \simeq \iota(P^{\F}_{\epsilon(V)})$,
where $\iota : \F \rightarrow \Fq$ is the functor given in
Theorem \ref{1.3}.
\item \label{PV0-2}
The functor $P_V^{(0)}$ is a direct summand of $P_V$.
\end{enumerate}
\end{thm}

Before proving this result, we give the following useful characterization of
the morphisms of rank zero, which is a straightforward consequence of
the definition of the rank of a morphism.
\begin{lm} \label{2.7}
Let $V$ be an object in $\Tq$. A morphism $T=[ V
\xrightarrow{\alpha} W \bot W' \xleftarrow{i_W} W]$ has rank zero if
and only if $ p_{W'} \circ \alpha$ is an injective linear map, where $p_{W'}$ is the orthogonal projection from $W \bot W'$ to $W'$.
\end{lm}

For $V$ and $W$ objects in $\Tq$, the forgetful functor $\epsilon: \Tq \rightarrow \E^f$ gives rise to a
map \linebreak $\mathrm{Hom}_{\Tq}(V,W) \rightarrow \mathrm{Hom}_{\E^f}(\epsilon (V),
\epsilon (W))$. By passage to the vector spaces freely generated by these sets and by functoriality of $\epsilon$, we deduce the existence of a morphism from $P_V$ to $\iota(P^{\F}_{\epsilon(V)})$. As the functors $P_V^{(0)}$ are subfunctors of $P_V$, we obtain a morphism $f$ from $P_V^{(0)}$ to $\iota(P^{\F}_{\epsilon(V)})$. Consequently, to prove Theorem \ref{2.6}, it is sufficient to prove
the following proposition.

\begin{prop} \label{2.10}
The map $P_V^{(0)}(W) \xrightarrow{f_W}
P^{\F}_{\epsilon(V)}(\epsilon(W))$ is an isomorphism for $V$ and $W$
objects in $\Tq$.
\end{prop}
The surjectivity of $f_W$ relies on the following lemma, which is an improved version of the fullness of the forgetful
functor $\epsilon$ given in \cite{math.AT/0606484}.
\begin{lm}
Let $(V,q_V)$ and $(W,q_W)$ be two objects of $\Tq$ and \linebreak $f \in
\mathrm{Hom}_{\E^f}(\epsilon(V,q_V), \epsilon (W,q_W))$ a linear map, then there exists a morphism \linebreak $T= [V
\xrightarrow{\varphi} W \bot Y \xleftarrow{i_W} W]$ of rank zero
such that $\epsilon(T)=f$.
\end{lm}
\begin{proof}
As the quadratic space $V$ is nondegenerate, we know that it has even
dimension. We write $\mathrm{dim}(V)=2n$. We prove the result by induction on
$n$.

To start the induction, let $(V,q_V)$ be a nondegenerate quadratic
space of dimension two, with symplectic basis $\{ a,b \}$ and $f: V
\rightarrow W$ be a linear map. The following linear
map preserves the quadratic form:
$$ \begin{array}{cccl}
g_1: & V& \rightarrow & W \bot H_1 \bot H_0  \bot H_0 \simeq W \bot
\mathrm{Span}(a_1,b_1)  \bot \mathrm{Span}(a_0,b_0) \bot \mathrm{Span}(a'_0,b'_0)\\
       & a & \longmapsto & f(a)+(q(a)+q(f(a))) a_1+a_0\\
       & b & \longmapsto & f(b)+(q(b)+q(f(a))) a_1+ (1+B(f(a), f(b))) b_0+a'_0.
\end{array}$$ 
Consequently, the morphism:
$T= V \xrightarrow{g_1} W \bot H_1 \bot H_0 \hookleftarrow W$,
is a morphism of rank zero of $\Tq$ such that $\epsilon(T)=f$.

Let $V_n$ be a nondegenerate quadratic space of dimension $2n$, $\{a_1,
b_1, \ldots, a_n, b_n \}$ be a symplectic basis of $V_n$ and $f_n: V_n
\rightarrow W$ be a linear map. By induction, there exists a  map:
$$ \begin{array}{cccc}
g_n: & V_n& \rightarrow & W \bot Y\\
       & a_i & \longmapsto & f_n(a_i)+y_i\\
       & b_i & \longmapsto & f_n(b_i)+z_i\\
 \end{array}$$
where $y_i$ and $z_i$, for all integers $i$ between $1$ and $n$, are
elements of $Y$. The map $g_n$ preserves the quadratic form and the morphism 
$T=[ V_n \xrightarrow{g_n} W \bot Y \hookleftarrow W]$
is of rank zero and verifies 
$\epsilon([V_n \xrightarrow{g_n} W \bot Y \hookleftarrow W])=f_n.$

Let $V_{n+1}$ be a nondegenerate quadratic space of dimension
$2(n+1)$, \linebreak
$\{a_1,
b_1, \ldots, a_n, b_n, a_{n+1}, b_{n+1} \}$ a symplectic basis of
$V_{n+1}$ and $f_{n+1}: V_{n+1} \rightarrow W$ a linear map. To define the
map $g_{n+1}$, we will consider the restriction of $f_{n+1}$ to
$V_n$ and extend the map $g_n$ given by the inductive assumption. For
that, we need the following space:
$E \simeq W \bot W' \bot H_0^{\bot n} \bot H_0^{\bot n} \bot H_1 \bot H_0 \bot
H_0 $
for which we specify the notations for a basis:
$$E \simeq W \bot W' \bot ( \bot_{i=1}^n \mathrm{Span}(a_0^i, b_0^i)) \bot ( \bot_{i=1}^n \mathrm{Span}(A_0^i, B_0^i)) 
  \bot \mathrm{Span}(A_1, B_1) $$
  $$\bot \mathrm{Span}(C_0, D_0) \bot
  \mathrm{Span}(E_0, F_0).$$
The following map: 
$$ \begin{array}{cccl}
g_{n+1}: & V& \rightarrow & W \bot W' \bot H_0^{\bot n} \bot H_0^{\bot
  n} \bot H_1 \bot H_0 \bot H_0\\
        & a_i & \longmapsto & f_{n+1}(a_i)+y_i+a_0^i  \quad \mathrm{for} \ i\ \mathrm{between}\ 1\ \mathrm{and}\ n\\
       & b_i & \longmapsto & f_{n+1}(b_i)+z_i+A_0^i\\
             & a_{n+1} & \longmapsto &
       f_{n+1}(a_{n+1})+(q(a_{n+1})+q(f_{n+1}(a_{n+1})))A_1+C_0\\
       &         &             &+\sum_{i=1}^n
       B(f_{n+1}(a_i),f_{n+1}(a_{n+1})) b_0^i\\
       &         &             &+\sum_{i=1}^n B(f_{n+1}(b_i),f_{n+1}(a_{n+1})) B_0^i \\
       & b_{n+1} & \longmapsto &
       f_{n+1}(b_{n+1})+(q(b_{n+1})+q(f_{n+1}(b_{n+1})))A_1\\
       &        &              &+(1+B(f_{n+1}(a_{n+1}),
       f_{n+1}(b_{n+1})))D_0\\
       &        &              &+\sum_{i=1}^n
       B(f_{n+1}(a_i), f_{n+1}(b_{n+1})) b_0^i\\
       &         &             &+\sum_{i=1}^n  B(f_{n+1}(b_i), f_{n+1}(b_{n+1})) B_0^i + E_0      
\end{array}$$
preserves the quadratic form. Furthermore, the morphism 
$$T=[ V_{n+1} \xrightarrow{g_{n+1}} W \bot W' \bot H_0^{\bot n} \bot H_0^{\bot
  n} \bot H_1 \bot H_0 \bot H_0 \hookleftarrow W]$$
is of rank zero and satisfies: $\epsilon(T)=f_{n+1}$, which completes the inductive step.

\end{proof}

The proof of the injectivity of $f_W$ relies on the following
 result, which can be regarded as Witt's theorem for
degenerate quadratic forms.

\begin{thm} \label{2.12}
Let $V$ be a nondegenerate quadratic space, $D$ and $D'$
subquadratic spaces (possibly degenerate) of $V$ and
$\underline{f}:D \rightarrow D'$ an isometry between these two
quadratic spaces. Then, there exists an isometry
$f: V \rightarrow V$ such that the following diagram is commutative:
 $$\xymatrix{
V  \ar[r]^f & V  \\
D \ar@{^{(}->}[u] \ar[r]_{\underline{f}}& D'. \ar@{^{(}->}[u] 
}$$
\end{thm}
\begin{proof}
For a proof of this result, we refer the reader to \cite{Bourbaki} \S4, theorem 1.
\end{proof}

\begin{proof}[Proof of the injectivity of $f_W$]
The natural map $f$ is induced by the natural map  
$$\mathrm{Hom}_{\Tq}^{(0)}(V,-) \rightarrow
\mathrm{Hom}_{\E^f}(\epsilon(V),\epsilon(-))$$
by passage to the vector spaces freely generated by these
sets. So, $f$ is injective if and only if this natural map
is injective. Consequently, it is sufficient to verify that, for $T=[V \xrightarrow{\alpha} W \bot W'
\xleftarrow{i_W} W ]$ and $T'=[V \xrightarrow{\alpha'} W \bot W''
\xleftarrow{i'_W} W ]$ two generators of $P_V^{(0)}(W)$ such that  
\begin{equation}
p_W \circ \mathcal{O}(\alpha)=p'_W \circ \mathcal{O}(\alpha'), \label{inj-1}
\end{equation} 
we have $T=T'$. 

Let $\{ a_1, b_1, \ldots, a_n, b_n \}$ be a symplectic
basis of $V$. We deduce from \ref{inj-1} that, for all $i \in \{ 1,
\ldots, n \}$, we have:
\begin{equation}
 \alpha(a_i)=w_i +w'_i ,\  \alpha(b_i)=x_i +x'_i \label{inj-2}
\end{equation}
and 
\begin{equation}
\alpha'(a_i)=w_i +w''_i ,\  \alpha'(b_i)=x_i +x''_i \label{inj-3}
\end{equation}
where, for all $i \in \{1, \ldots, n \}$, $w_i$ and $x_i$ are in $W$, $w'_i$ and $x'_i$ are in $W'$ and $x''_i$ and $w''_i$ are in $W''$. By Lemma \ref{2.7}, since the morphisms are of rank zero, $\{ w'_1, x'_1, \ldots, w'_n, x'_n \}$ and $\{ w''_1, x''_1, \ldots,
w''_n, x''_n \}$ are two linearly  independent families of vectors.

We will denote by $\underline{W'}=\mathrm{Span}(w'_1, x'_1, \ldots, w'_n, x'_n
)$ (respectively \linebreak $\underline{W''}=\mathrm{Span}(w''_1, x''_1, \ldots, w''_n, x''_n
)$) the subquadratic space (possibly degenerate), of $W'$
(respectively $W''$) and we define the linear map $\underline{f}:
\underline{W'} \rightarrow \underline{W''}$ by
$\underline{f}(w'_i)=w''_i$ and $\underline{f}(x'_i)=x''_i$ for all
$i \in \{1, \ldots, n \}$.

Since $\alpha$ and $\alpha'$ preserve the quadratic forms, we deduce
from the relations \ref{inj-2} and \ref{inj-3} that $\underline{f}$
preserves the quadratic form. Hence, we can apply Theorem
\ref{2.12} to the nondegenerate space $W' \bot W''$, which gives a morphism $f: W' \bot W'' \rightarrow W'
\bot W''$ of $\Eq$, such that, the restriction of this morphism to $W'$ coincides with $\underline{f}$. We deduce the
commutativity of the following diagram:

$$\xymatrix{
    & W \ar@{^{(}->}[d]^{i_W} \ar@{^{(}->}[ddr]^{i'_W} & \\
V \ar[r]^(.3){\tilde{\alpha}} \ar[drr]_{\tilde{\alpha'}} & W \bot(W' \bot W'')
    \ar[rd]^(.4){\mathrm{Id} \bot f} & \\
  &  &  W \bot(W' \bot W'')
}$$
where $\tilde{\alpha}=i_{W \bot W'} \circ \alpha$ and $\tilde{\alpha'}=i_{W \bot W''} \circ \alpha'$.
Consequently, we obtain the equality $T=T'$ since, by inclusion, we have
$$T=[V \xrightarrow{\alpha} W \bot W' \xleftarrow{i_W} W ]=[V
\xrightarrow{\tilde{\alpha}} W \bot W' \bot W'' \xleftarrow{i_W} W ]$$ and $$T'=[V \xrightarrow{\alpha'} W \bot W'' \xleftarrow{i'_W} W ]=[V
\xrightarrow{\tilde{\alpha'}} W \bot W' \bot W'' \xleftarrow{i'_W} W ].$$

\end{proof}

\begin{nota} \label{2.13}
For $V$ and $W$ two objects of \Eq, and $f$ a
morphism of $\mathrm{Hom}_{\E^f}(\epsilon(V), \epsilon (W))$,
we denote by $t_f$ the morphism of $\mathrm{Hom}^{(0)}_{\Tq}(V,W)$
corresponding to $f$ and by $[t_f]$ the canonical generator of
$P_V^{(0)}(W)$ obtained from $t_f$. To simplify the notation, we will denote the morphism $t_{\mathrm{Id}_V}$ of $\mathrm{Hom}^{(0)}_{\Tq}(V,V)$ by $e_V$.
\end{nota}

%\subsection{Le foncteur $P_V^{(0)}$ est facteur direct de $P_V$}
We deduce from the first point of Theorem \ref{2.6} the following corollary.
\begin{cor} \label{2.14}
For $V$, $W$ and $X$ objects of \Eq, $f: \epsilon(W) \rightarrow \epsilon(X)$ and $g:
\epsilon(V) \rightarrow \epsilon(W)$ morphisms of $\E^f$, we have:
$$t_f \circ t_g =t_{f \circ g}$$
where $t_f,\ t_g$ and $t_{f \circ g}$ are respectively the morphisms of $\mathrm{Hom}^{(0)}_{\Tq}(W,X)$, $\mathrm{Hom}^{(0)}_{\Tq}(V,W)$ and
$\mathrm{Hom}^{(0)}_{\Tq}(V,X)$ associated to the linear maps $f,\ g$ and $f \circ g$.
\end{cor}

We can apply this result to the idempotents of the ring of
endomorphisms $\mathrm{End}(P_V)$, to obtain the following proposition.
\begin{prop} \label{2.15}
The canonical generator $[e_V]$ of $P_V^{(0)}$ is an idempotent of the
ring of endomorphisms $\mathrm{End}(P_V)$
such that $P_V.[e_V] \simeq P_V^{(0)}.$
\end{prop}
\begin{proof}
The  canonical generator $[e_V]$ is an idempotent of $\mathrm{End}(P_V)$
by Corollary \ref{2.14}. By definition of the rank
filtration $P_V.[e_V] \subset P_V^{(0)}$ and, for a canonical generator $[t_f]$ of $P_V^{(0)}$, we
have $[t_f]=[t_f] \cdot [e_V]$.
\end{proof}

The idempotent  $[e_V]$ plays a central r\^{o}le in the proof of the
thickness of the subcategory $\iota(\F)$ in $\Fq$, which is the
subject of the following paragraph. For that, the following result is necessary. 

\begin{lm} \label{2.18}
Let $V$ and $W$ be objects of $\Tq$, the functor $\iota$ induces an
isomorphism:
$$\mathrm{Hom}_{\F}(P_{\epsilon(V)}^{\F}, P_{\epsilon(W)}^{\F}) \xrightarrow{\simeq}\mathrm{Hom}_{\Fq}(P_V^{(0)}, P_W^{(0)}), $$
where $\epsilon: \Tq \rightarrow \E$ is the forgetful functor.
\end{lm}
\begin{proof}
By Proposition \ref{2.15} and Theorem \ref{2.6} we have the following equivalences:
$$\mathrm{Hom}_{\Fq}(P_V^{(0)}, P_W^{(0)})\simeq
P_W^{(0)}(e_V) P_W^{(0)}(V)\simeq P_W^{(0)}(V)$$
$$\simeq\iota
(P_{\epsilon(W)}^{\F})(V) \simeq P_{\epsilon(W)}^{\F}(\epsilon(V))\simeq\mathrm{Hom}_{\F}(P_{\epsilon(V)}^{\F}, P_{\epsilon(W)}^{\F}).$$
\end{proof}

To conclude this paragraph, we give the following property of $e_V$ which will be useful in section \ref{4} concerning the polynomial functors of $\Fq$.

\begin{lm} \label{2.17}
For $V$ and $W$ two objects of $\Tq$, we have: $e_{V \bot W}=e_V \bot e_W$,
where $\bot: \Tq \times \Tq \rightarrow \Tq$ is the functor induced by
the orthogonal sum.
\end{lm}

\begin{proof}
This is a straightforward consequence of Proposition \ref{2.10}.
\end{proof}

\subsubsection{The category $\iota(\F)$ is a thick subcategory of  $\Fq$}

The aim of this paragraph is to prove the following result.
\begin{thm} \label{2.19}
The category $\iota(\F)$ is a thick subcategory of 
  $\Fq$, where \linebreak  $\iota : \F \rightarrow \Fq$ is the functor
  defined in Theorem \ref{1.3}.
\end{thm}
To prove this theorem, we need the following general result about the
precomposition functor which is proved in the Appendix of \cite{math.AT/0606484}.
\begin{prop} \label{2.20}
Let $\C$ and $\D$ be two small categories, $\A$ be an abelian category,
$F: \C \rightarrow \D$ be a functor and $- \circ F: \mathrm{Func}(\D, \A)
\rightarrow \mathrm{Func}(\C, \A)$ be the precomposition functor, where
$\mathrm{Func}(\C, \A)$ is the category of functors from $\C$ to $\A$.
If $F$ is full and essentially surjective, then
any subobject (respectively quotient) of an object in the image of the precomposition functor
is isomorphic to an object in the image of the precomposition functor.
\end{prop}

\begin{proof}[Proof of Theorem \ref{2.19}]
\begin{itemize}
\item The subcategory $\iota(\F)$ of $\Fq$ is full by Theorem \ref{1.3}.
\item Let $F^{\F}$ be an object in $\F$ and $G$ a subobject of $\iota(F^{\F}) $. Let $\F'$ be the category of functors
from $\E^{f-(even)}$ to $\E$, where $\E^{f-(even)}$ is the full subcategory of $\E^f$ having as
objects the $\FF$-vector spaces of even dimension. The categories $\F$ and $\F'$ are equivalent \cite{math.AT/0606484}. The functor $\epsilon:\Eq \rightarrow \E^{f}$ factorizes through the inclusion $
\E^{f-(even)} \hookrightarrow \E^{f}$. This induces a functor
$\epsilon': \Eq \rightarrow \E^{f-(even)}$ which is full and
essentially surjective. Consequently, we can use Proposition \ref{2.20} to obtain: $G \simeq \iota(G^{\F})$. Similarly, we obtain the result for the quotient.

\item
Let $G^{\F}$ and $H^{\F}$ be objects of $\F$, we set $G= \iota(G^{\F})$ and $H=\iota(H^{\F})$. For a short exact sequence: $0 \rightarrow G \rightarrow F \rightarrow H \rightarrow 0,$
we have to prove that there exists a functor $F^{\F}$ in $\F$
such that $F=\iota(F^{\F})$.

Let $ P_1 \rightarrow P_0 \rightarrow G^{\F}
\rightarrow 0$ and $  Q_1 \rightarrow Q_0 \rightarrow H^{\F}
\rightarrow 0$
be projective presentations of $G^{\F}$ and $H^{\F}$ in $\F$, we have
the following commutative diagram
$$\xymatrix{
     0 \ar[r]   &   \iota(P_1)  \ar[r] \ar[d]  &   \iota(P_1) \oplus \iota(Q_1) \ar[r] \ar[d]  &
  \iota(Q_1) \ar[r] \ar[d]  & 0   \\
0 \ar[r]   &   \iota(P_0)  \ar[r] \ar[d]  &   \iota(P_0) \oplus \iota(Q_0) \ar[r] \ar[d]  &
  \iota(Q_0) \ar[r]  \ar[d] & 0   \\
0 \ar[r]   &   G \ar[d]  \ar[r]  &  F \ar[d] \ar[r]  &
  H \ar[d] \ar[r]   & 0   \\
            &  0 & 0 & 0}$$
where the columns are projective resolutions in $\Fq$, by the horseshoe
lemma. By Lemma \ref{2.18}, the morphism $\iota(P_1) \oplus
     \iota(Q_1) \rightarrow \iota(P_0) \oplus \iota(Q_0)$ is induced
     by a morphism of $\F$ denoted by $f$. Consequently, \linebreak $F \cong \iota(\mathrm{Coker}(f)) \in \iota(\F)$.

\end{itemize}
\end{proof}
By Theorem \ref{2.19}, we deduce from Lemma \ref{2.18}, the following characterization of the simple
functors of $\F$ in $\Fq$ which will be used in section 
\ref{4} of this paper concerning the polynomial functors of $\Fq$.

\begin{lm} \label{2.21}
\begin{enumerate}
\item
Let $F$ be a functor of $\Fq$, then $F$ is in the image of the functor
$\iota: \F \rightarrow \Fq$ if and only if, for all objects $V$ in $\Tq$,
$$F(e_V) F(V)=F(V).$$
\item
Let $S$ be a simple object in $\Fq$, then $S$ is in the image of the
functor $\iota: \F \rightarrow \Fq$ if and only if there exists an
object $W$ in $\Tq$ such that 
$$S(e_W) S(W) \ne 0.$$
\end{enumerate}
\end{lm}
\begin{proof}
\begin{enumerate}
\item
The forward implication is a consequence of the following fact: for a functor
$F$ in the image of Ä $\iota$, $\mathrm{Hom}_{\Fq}(P_V^{(0)},F)=F(e_V) F(V)=F(V).$
The reverse implication relies on the fact that the condition \linebreak  $F(e_V)
F(V)=F(V)$ implies that $F$ is a quotient of a sum of projective
objects of the form $P_V^{(0)}$. Since the category $\iota(\F)$ is
thick in $\Fq$ by Theorem \ref{2.19}, we obtain the result.

\item
Observe that, if $S(e_W)S(W) \ne
0$, we have $\mathrm{Hom}_{\Fq}(P_W^{(0)},S) \ne 0$, thus $S$ is a
quotient of $P_W^{(0)}$ by simplicity of $S$. Lemma \ref{2.18}
implies that there exists a one-to-one correspondance between the
indecomposable factors of $P_V^{(0)}$ and those of
$P_{\epsilon(V)}^{\F}$. We deduce that the simple quotients of
$P_V^{(0)}$ arise from $\F$. Consequently, $S$ is in the image of the functor $\iota$.
\end{enumerate}
\end{proof}

\subsubsection{The quotient $P_V/P_V^{(\mathrm{dim}(V)-1)}$}
The aim of this paragraph is to prove the following result:

\begin{prop} \label{2.22}
Let $V$ be an object in \Tq, we have a natural equivalence:
$$P_V / P_V^{(\mathrm{dim}(V)-1)} \simeq \kappa(\mathrm{iso}_V)$$
where $\mathrm{iso}_V$ is an isotropic functor and $\kappa: \Gq
\rightarrow \Fq$ is the functor given in Theorem \ref{1.4}.
\end{prop}

To prove this proposition, we need the following notation and result:

\begin{nota}
Denote by $\sigma_f$ the natural map $P_V
\xrightarrow{\sigma_f} \kappa(\mathrm{iso}_V)$  which corresponds to the canonical generator $[V \xleftarrow{\mathrm{Id}}  V
\xrightarrow{f} V]$ of $\mathrm{iso}_V(V)$ by the equivalence $\mathrm{Hom}(P_V,\kappa(\mathrm{iso}_V)) \simeq \mathrm{iso}_V(V) \simeq \FF[O(V)]$ given by the Yoneda lemma.
\end{nota}

\begin{lm} \label{2.23}
The functor $\kappa(\mathrm{iso}_V)$ of $\Fq$  is a quotient of the
functor $P_V=\FF[\mathrm{Hom}_{\Tq}(V,-)]$.
\end{lm}
\begin{proof}
The natural map  $P_V
\xrightarrow{\sigma_f} \kappa(\mathrm{iso}_V)$ is surjective: a pre-image of the
canonical generator $[V \xleftarrow{\mathrm{Id}} V
\xrightarrow{g} W]$ of $\kappa(\mathrm{iso}_V)(W)$ by $(\sigma_f)_W$,
is the morphism \linebreak $[V \xrightarrow{g \circ f^{-1}} W \xleftarrow{\mathrm{Id}} W]$.
\end{proof}
A formal consequence of the previous lemma is given in the
following result.
\begin{lm} \label{2.24}
The functor $\kappa(\mathrm{iso}_V)$ of $\Fq$ is a quotient of the functor
$P_V/ P_V^{(\mathrm{dim}(V)-1)}$.
\end{lm}

\begin{proof} 
By definition of the filtration and by the previous lemma, we
have the diagram:
$$\xymatrix{
0 \ar[r] &  P_V^{(\mathrm{dim}(V)-1)} \ar[r]^(.7){i} & P_V \ar@{>>}[d]^{\sigma_{\mathrm{Id}}}
\ar@{>>}[r] & P_V/ P_V^{(\mathrm{dim}(V)-1)} \ar[r] \ar@{.>>}[dl]^{\tau}& 0\\
 & &\kappa(\mathrm{iso}_V) & & }$$
where $i$ is the canonical inclusion of $P_V^{(\mathrm{dim}(V)-1)}$ in $P_V$.
By definition of $\sigma_{\mathrm{Id}}$, we have $\sigma_{\mathrm{Id}}
\circ i=0$, from which we deduce the existence of the surjection $\tau : P_V/ P_V^{(\mathrm{dim}(V)-1)} \rightarrow \kappa(\mathrm{iso}_V) .$
\end{proof}
We will prove below that this natural map is an
isomorphism. It is sufficient to prove the following result.

\begin{prop} \label{2.25}
For $V$ and $W$ two objects of \Tq, we have an isomorphism
$$(P_V/ P_V^{(\mathrm{dim}(V)-1)})(W) \simeq \kappa(\mathrm{iso}_V)(W).$$
\end{prop}

The proof of this proposition relies on the following lemma.
\begin{lm} \label{2.26}
For a non-zero canonical generator of  $(P_V/ P_V^{(\mathrm{dim}(V)-1)})(W)$ represented by the morphism $T=[V \xrightarrow{g} W \bot W' \xleftarrow{i_W} W
] $ of $\Tq$, we have $g(V) \subset W$ and $T=[V \xrightarrow{f} W  \xleftarrow{\mathrm{Id}} W ]$, where $g=i_W \circ f$.
\end{lm}

\begin{proof}
By definition of the filtration, for $V$ and $W$ two objects of \Tq,
the vector space $(P_V/ P_V^{(\mathrm{dim}(V)-1)})(W)$ is generated by
 $\mathrm{Hom}_{\Tq}^{[\mathrm{dim}(V)]}(V,W)$
where $\mathrm{Hom}_{\Tq}^{[\mathrm{dim}(V)]}(V,W)$ is the set of
morphisms from $V$ to $W$ whose the pullback $D$ in $\Eqd$ is a
quadratic space such that $\mathrm{dim}(D)=\mathrm{dim}(V)$. We deduce
from the existence of a monomorphism from $D$ to $V$ and from the
equality of the dimensions, that $D$ and $V$ are
isometric. Consequently, for the morphism $T$ of the statement, we
have, by definition of the pullback, $g(V) \subset W$. Thus, we have $T=[V \xrightarrow{f} W
\xleftarrow{\mathrm{Id}} W ]$, where $g=i_W \circ f$, by the equivalence relation defined
over the morphisms of $\Tq$ in Definition \ref{1.1}.

\end{proof}

\begin{proof}[Proof of Proposition \ref{2.25}]
The natural map $\tau$ obtained in the proof of Lemma \ref{2.24}
defines, for $W$ an object in $\Tq$, the linear map
$$ \begin{array}{cccc}
\tau_W : & (P_V/ P_V^{(\mathrm{dim}(V)-1)})(W) & \rightarrow & \kappa(\mathrm{iso}_V)(W)\\
       & T=[V \xrightarrow{f} W  \xleftarrow{\mathrm{Id}} W ] &
       \longmapsto & [V \xleftarrow{\mathrm{Id}} V \xrightarrow{f} W]
\end{array}$$
which is clearly an isomorphism.
\end{proof}

\section{Decomposition of the standard projective functors $P_{H_0}$
  and $P_{H_1}$} \label{3}

On abelian categories, the
decompositions into direct summands of a functor $F$ of $\Fq$ correspond to
decompositions into orthogonal idempotents of $1$ in the ring
$\mathrm{End}_{\Fq}(F)$ (see for example \cite{Curtis-Reiner-I}). One of the difficulties of the category $\Fq$ lies in the fact that
the rings of endomorphisms of projectives $P_V$ and their
representations are not well-understood. The decompositions of projectives $P_V$, obtained in work in preparation, using a refinement of the rank filtration will allow us to understand the structure of these rings  better.

In this section, we obtain the decompositions into indecomposable
factors of the projective objects $P_{H_0}$ and $P_{H_1}$ by an
explicit study of the filtration defined in section \ref{2}. This section concludes by several consequences of these
decompositions. In particular, we give a classification of the ``small''
simple functors of $\Fq$, which is an essential ingredient in the following
section about the polynomial functors of $\Fq$. 

\subsection{Decomposition of $P_{H_0}$}

To obtain the decomposition of the functor
$P_{H_0}$ into indecomposable factors, we give an explicit description of the subquotients
of the filtration; then, we prove that the filtration splits for
this functor and we identify the factors of this decomposition.

\subsubsection{Explicit description of the subquotients of the filtration}

The aim of this paragraph is to give a basis of the vector spaces
$P_{H_0}^{(0)} (V)$, $P_{H_0}^{(1)}/P_{H_0}^{(0)} (V)$  and
$P_{H_0}/P_{H_0}^{(1)} (V)$ for $V$ a given object in \Tq.

We deduce from Theorem \ref{2.6} and Notation \ref{2.13}, the
following result.
\begin{lm} \label{3.1}
A basis $\mathcal{B}_{H_0}^{(0)}$ of $P_{H_0}^{(0)} (V)$ is given by
the set:
$$\mathcal{B}_{H_0}^{(0)}=\{ [t_f] \  \mathrm{for\ } f \in
\mathrm{Hom}_{\E^f}(\FF^{\oplus 2}, \epsilon(V)) \}.$$
\end{lm}

By definition of the filtration, a canonical generator of  $P_{H_0}^{(1)}/P_{H_0}^{(0)} (V)$, represented by the morphism $T=[H_0 \xrightarrow{f}
V \bot L \xleftarrow{i_V} V]$ of $\Tq$,
satisfies the following property: $I=f(H_0) \cap i(V)$ is a quadratic space of dimension one.

\begin{lm} \label{3.2}
Let $T=[H_0 \xrightarrow{f} V \bot L \xleftarrow{i_V} V]$ be a morphism of $\Tq$ which represents a
canonical generator of $P_{H_0}^{(1)}/P_{H_0}^{(0)} (V)$, and $\{ a_0, b_0 \}$
be a symplectic basis of $H_0$, then the map $f$ in $T$
has one of the three following forms.
\begin{enumerate}
\item If $I=(f(a_0),0)$, the map $f: H_0 \rightarrow V \bot L$ is
  defined by:
$$f(a_0)=v \mathrm{\qquad and \qquad} f(b_0)=w+l$$
for $v$ and $w$ elements of $V$ satisfying $q(v)=0$ and $B(v,w)=1$
and $l$ a non-zero element of $L$.
\item If $I=(f(b_0),0)$, the map $f: H_0 \rightarrow V \bot L$ is
  defined by:
$$f(a_0)=v+l \mathrm{\qquad and \qquad} f(b_0)=w$$
for $v$ and $w$ elements of $V$ satisfying $q(w)=0$ and $B(v,w)=1$
and $l$ a non-zero element of $L$.
\item If $I=(f(a_0+b_0),1)$, the map $f: H_0 \rightarrow V \bot L$ is
  defined by:
$$f(a_0)=v+l \mathrm{\qquad and \qquad} f(b_0)=w+l$$
for $v$ and $w$ elements of $V$ satisfying $q(v+w)=1$ and $B(v,w)=1$
and $l$ a non-zero element of $L$.
\end{enumerate}
\end{lm}

\begin{proof}
The quadratic space $H_0$ has three subspaces of dimension one which are:
$\mathrm{Span}(a_0)$ and $\mathrm{Span}(b_0)$ isometric to $(x,0)$ and
$\mathrm{Span}(a_0+b_0)$ isometric to $(x,1)$. These three subspaces
give rise to each one of the maps $f$ defined in the statement.
\end{proof}
\begin{nota} \label{3.3}
The morphisms $[H_0 \xrightarrow{f} V \bot L \xleftarrow{i_V} V]$, where
$f$ is one of the morphisms described in the point $(1)$ (respectively $(2)$
and $(3)$) of the previous lemma, will
be known as type $A$ (respectively $B$ and $C$) morphisms.
\end{nota}
We have the following proposition.

\begin{prop}\label{3.4}
For $T=[H_0 \xrightarrow{f} V \bot L \xleftarrow{i_V} V]$ and
$T'=[H_0 \xrightarrow{f'} V \bot L' \xleftarrow{i'_V} V]$ morphisms of $\Tq$ which represent canonical generators
of $P_{H_0}^{(1)}/P_{H_0}^{(0)} (V)$, the following properties are equivalent.
\begin{enumerate}
\item The morphisms $T$ and $T'$ of
$\mathrm{Hom}_{\Tq}(H_0,V)$ have the same type and satisfy the relation $p_V \circ f= p'_V \circ f'$.
\item The morphisms $T$ and $T'$  of $\mathrm{Hom}_{\Tq}(H_0,V)$ are equal.
\end{enumerate} 
\end{prop}
The proof of the implication $(2) \Rightarrow (1)$ relies on the
following technical lemma.
\begin{lm} \label{3.5}
Let $T=[V \xrightarrow{g} X \xleftarrow{h} W]$ and $T'=[V
\xrightarrow{g'} X' \xleftarrow{h'} W']$ be morphisms of
$\mathrm{Hom}_{\Tq}(V,W)$. If $T=T'$, then $g(V) + h(W) \simeq g'(V)+ h'(W)$ in $\Eqd$.
\end{lm}

\begin{proof}
By definition of the equivalence relation given in Definition
\ref{1.1}, it is sufficient to prove that, for two morphisms $T$ and
$T'$ such that $T \mathcal{R} T'$, we have  $g(V) + h(W) \simeq g'(V) + h'(W)$. 

By definition, $g(V) + h(W)$ is the smallest, possibly degenerate, quadratic space such
that we have a commutative diagram in $\Eqd$, of the form:
$$\xymatrix{
  & W \ar[d]_{\beta} \ar[rdd]^{h}\\
V \ar[r]^-{\alpha}  \ar[drr]_{g} & g(V) + h(W) \ar[dr]^(.4){\gamma }\\
  &   &X.\\
}$$
Similarly, $g'(V) + h'(W)$ is the smallest quadratic space such
that we have an analogous commutative diagram.
By definition of the relation $\mathcal{R}$, we have the existence of
a morphism $\delta$ in $\Eq$ such that the following diagram is commutative:
$$\xymatrix{
  & W \ar[d]_{h} \ar[rdd]^{h'}\\
V \ar[r]^{g}  \ar[drr]_{g'} & X \ar[dr]^(.4){\delta }\\
  &   &X'.\\
}$$
By the consideration of the following commutative diagram in $\Eqd$
$$\xymatrix{
  & W \ar[d]_-{\beta} \ar[rdd]_-{h} \ar[rrddd]^-{h'}\\
V \ar[r]^-{\alpha}  \ar[drr]^-{g} \ar[ddrrr]_-{g'}& Y \ar[dr]\\
  &   &X \ar[dr]^(.4){\delta }\\
  &   &   & X'
}$$
where $Y=g(V) + h(W)$, we deduce from the minimality of $g'(V) + h'(W)$, the existence of
a morphism in $\Eqd$ from $g'(V) + h'(W)$ to $g(V) + h(W)$ such
that the corresponding diagram is commutative. Then, by minimality of
$g(V) + h(W)$ for $T$, we obtain: $g(V) + h(W) \simeq g'(V) + h'(W)$.
\end{proof}

\begin{proof}[Proof of Proposition \ref{3.4}]
Suppose that the morphisms $T$ and $T'$ of
$\mathrm{Hom}_{\Tq}(H_0,V)$ are of type $A$
such that $p_V \circ f= p'_V \circ f'$. We deduce that:
$$f(a_0)=v;\  f(b_0)=w+l \mathrm{\ and\ } f'(a_0)=v;\  f'(b_0)=w+l'$$
for $v$ and $w$ elements of $V$ and $l$ (resp. $l'$) a
non-zero element of $L$ (resp. $L'$).

Since the maps $f$ and $f'$ preserve the quadratic forms, we have \linebreak
$q(b_0)=q(w)+q(l)=q(w)+q(l')$. We deduce that $q(l)=q(l')$, thus the map $\mathrm{Span}(l)
\xrightarrow{\underline{\alpha}}  \mathrm{Span}(l')$, such that $\underline{\alpha}(l)=l'$, preserves the
quadratic form. Consequently, we can apply Theorem \ref{2.12} to
obtain the existence of a map $\alpha: L \bot L' \rightarrow L \bot
L'$ such that the following diagram is commutative:
$$\xymatrix{
L \bot L' \ar[r]^{\alpha} & L \bot L' \\
 \mathrm{Span}(l)  \ar@{^{(}->}[u] \ar[r]_{\underline{\alpha}}& \mathrm{Span}(l'). \ar@{^{(}->}[u] 
}$$
We deduce the commutativity of the diagram:
$$\xymatrix{
    & V \ar@{^{(}->}[d]^{i_V} \ar@{^{(}->}[ddr]^{i'_V} & \\
H_0 \ar[r]^-{f} \ar[drr]_{f'} & V \bot(L \bot L')
    \ar[rd]^-{\mathrm{Id} \bot \alpha} & \\
  &  &  V \bot(L \bot L').
}$$
Since
$T=[H_0 \xrightarrow{f} V \bot L \bot L' \xleftarrow{i_V} V ]$ and
    $T'=[H_0 \xrightarrow{f'} V \bot L \bot L' \xleftarrow{i'_V} V ]$,
    by inclusion, we deduce from the previous diagram that $T=T'$.

We reason in the same way, for the morphisms of type $B$ and $C$.

Conversely, if $T=T'$, by Lemma \ref{3.5} $f(H_0) + i_V(V) \simeq
f'(H_0) + i'_V(V)$. Consequently we deduce from Theorem
\ref{2.12} the existence of an isometry 
$\beta: V \bot L \bot L' \rightarrow V \bot L \bot L'$ making the
following diagram commutative:
 $$\xymatrix{
V \bot L \bot L'  \ar[r]^{\beta} & V \bot L \bot L'  \\
f(H_0) + i_V(V) \ar@{^{(}->}[u] \ar[r]_{\simeq}& f'(H_0) + i'_V(V). \ar@{^{(}->}[u] 
}$$
This yields the commutativity of the following diagram:
$$\xymatrix{
    & V \ar@{^{(}->}[d]^{i_V} \ar@{^{(}->}[ddr]^{i'_V} & \\
H_0 \ar[r]^{f} \ar[drr]_{f'} & V \bot L \bot L'
    \ar[rd]^{\beta} & \\
  &  &  V \bot L \bot L'
}$$
which implies that $\beta \circ i_V =i'_V$. Thus, $\beta=\mathrm{Id}_V \bot
\beta'$ where $\beta': L \bot L' \rightarrow L \bot L'$ is a morphism
of $\mathrm{Hom}_{\Eq}(L \bot L',L \bot L')$. Consequently, we have
$$f(a_0)=v+l;\  f(b_0)=w+l' \mathrm{\ and\ } f'(a_0)=v+\beta'(l);\
f'(b_0)=w+\beta'(l')$$
and we deduce that $p_V \circ f=p'_V$. Furthermore, since $\beta'$ is
inversible, for all $x$ in $L \bot L'$ we have: $x$ is non-zero if and
only if $\beta'(x)$ is non-zero. Consequently, $T$ and $T'$ have the
same type.
\end{proof}
This proposition justifies the following notation.
\begin{nota}
We will denote by $A_{v,w}$, $B_{v,w}$ and $C_{v,w}$ the morphisms of
$\mathrm{Hom}_{\Tq}(H_0,V)$ respectively of type $A$, $B$ and
$C$ and such that $p_V \circ f(a_0)=v$ and $p_V \circ f(b_0)=w$.
\end{nota}
The following result is a straightforward consequence of Lemma
\ref{3.2} and Proposition \ref{3.4}.
\begin{lm} \label{3.7}
A basis $\mathcal{B}_{H_0}^{(1)}$ of $P_{H_0}^{(1)}/P_{H_0}^{(0)} (V)$
is given by the set:
$$\begin{array}{ll}
\mathcal{B}_{H_0}^{(1)}=\{ & [A_{v,w}] \mathrm{\ for\ } v \mathrm{\ and\
  } w \mathrm{\ elements\ of\ } V \mathrm{\
    satisfying\ } q(v)=0 \mathrm{\ and\ } B(v,w)=1, \\
                        & [B_{v,w}] \mathrm{\ for\ } v \mathrm{\ and\
  } w \mathrm{\ elements\ of\ } V \mathrm{\
    satisfying\ } q(w)=0 \mathrm{\ and\ } B(v,w)=1, \\
                         &  [C_{v,w}] \mathrm{\ for\ } v \mathrm{\ and\
  } w \mathrm{\ elements\ of\ } V \mathrm{\
    satisfying\ } q(v+w)=1 \mathrm{\ and\ } B(v,w)=1 \}
\end{array}$$ 
\end{lm}

By Proposition \ref{2.22}, we have $(P_{H_0}/
P_{H_0}^{(1)})(V) \simeq  \kappa(\mathrm{iso}_{H_0})(V)$. We deduce
the following result.
\begin{lm} \label{3.8}
A basis $\mathcal{B}_{H_0}^{(2)}$ of $P_{H_0}/P_{H_0}^{(1)} (V)$ is
given by the set:
$$\mathcal{B}_{H_0}^{(2)}=\{ [D_f] \  \mathrm{for\ } f \in
\mathrm{Hom}_{\Eq}(H_0, V) \},$$ 
where $D_f$ is the morphism of $\Tq$ represented by the diagram: $H_0 \xrightarrow{f} V
\xleftarrow{\mathrm{Id}} V$.
\end{lm}

We end this paragraph by the rules of composition for the
morphisms $t_f$, $A_{v,w}$, $B_{v,w}$, $C_{v,w}$ and $D_f$, summarized
in the following proposition. This technical result will be
fundamental in the following paragraph, to prove the splitting of the filtration.
\begin{lm} \label{3.9}
Let $T=V \xrightarrow{\varphi} W \bot L \xleftarrow{i_W}
W$ be a morphism of $\mathrm{Hom}_{\Tq}(V,W)$. The following relations are satisfied:
\begin{enumerate}
\item For $f$ a morphism of  $\mathrm{Hom}_{\E^f}(\FF^{\oplus 2},
  \epsilon(V))$ we have: 
$$T \circ t_f=t_{\varphi \circ f}.$$
\item
\begin{enumerate}
\item For $v$ and $w$ elements of $V$ satisfying $q(v)=0$ and
  $B(v,w)=1$, we have:
$$T \circ A_{v,w}=\left\lbrace
\begin{array}{cc}
A_{\varphi(v), p_W \circ \varphi(w)}&\mathrm{if\ } \varphi(v) \in W\\
t_{p_W \circ (\varphi \bot \mathrm{Id} ) \circ \alpha} & \mathrm{otherwise.}
\end{array}
\right.$$
\item For $v$ and $w$ elements of $V$ satisfying $q(w)=0$ and
  $B(v,w)=1$, we have:
$$T \circ B_{v,w}=\left\lbrace
\begin{array}{cc}
B_{p_W \circ \varphi(v), \varphi(w)}&\mathrm{if\ } \varphi(w) \in W\\
t_{p_W \circ (\varphi \bot \mathrm{Id} ) \circ \alpha} & \mathrm{otherwise.}
\end{array}
\right.$$
\item For $v$ and $w$ elements of $V$ satisfying $q(v+w)=1$ and
  $B(v,w)=1$, we have:
$$T \circ C_{v,w}=\left\lbrace
\begin{array}{cc}
C_{p_W \circ \varphi(v),p_W \circ \varphi(w)}&\mathrm{if\ } \varphi(v+w) \in W\\
t_{p_W \circ (\varphi \bot \mathrm{Id} ) \circ \alpha} & \mathrm{otherwise.}
\end{array}
\right.$$
\end{enumerate}
\item \label{calculs-D}  For $f$ a morphism of
  $\mathrm{Hom}_{\Eq}(H_0, V)$, we have: 
$$T \circ D_f=\left\lbrace
\begin{array}{cl}
D_{\varphi \circ f} &\mathrm{if\ } \varphi \circ f(a_0) \in W
\mathrm{\ and\ } \varphi \circ f(b_0) \in W\\
A_{\varphi \circ f(a_0), p_W \circ \varphi \circ f(b_0)}&\mathrm{if\ } \varphi \circ f(a_0) \in W
\mathrm{\ and\ } \varphi \circ f(b_0) \notin W\\
B_{ p_W \circ \varphi \circ f(a_0), \varphi \circ f(b_0)}&\mathrm{if\ } \varphi \circ f(a_0) \notin W
\mathrm{\ and\ } \varphi \circ f(b_0) \in W\\
C_{ p_W \circ \varphi \circ f(a_0), p_W \circ \varphi \circ f(b_0)}&\mathrm{if\ } \varphi \circ f(a_0) \notin W
\mathrm{\ and\ } \varphi \circ f(b_0) \notin W \\
&\mathrm{\ and\ } \varphi
\circ f(a_0+b_0) \in W\\
t_{p_W \circ (\varphi \bot \mathrm{Id} ) \circ \alpha}&\mathrm{if\ } \varphi \circ f(a_0) \notin W
\mathrm{\ and\ } \varphi \circ f(b_0) \notin W\\ 
&\mathrm{\ and\ } \varphi
\circ f(a_0+b_0) \notin W.
\end{array}
\right.$$
\end{enumerate}

\end{lm}

\begin{proof}
By definition of the composition in $\Tq$, we have the following diagram:
$$\xymatrix{
   &   & W \ar@{^{(}->}[d] \\
   & V \ar[r]^{\varphi} \ar@{^{(}->}[d]& W \bot L \ar@{^{(}->}[d]\\
H_0 \ar[r]^-{\alpha}  & V \bot L' \ar[r]_-{\varphi \bot \mathrm{Id}} & W \bot L \bot L'\\
}\\$$
\begin{enumerate}
\item
For a morphism $t_f$, we have $\alpha(a_0)=f(a_0)+l$ and
$\alpha(b_0)=f(b_0)+m$. where $\{ l,m \}$ is a linearly independent
family of $L'$. Consequently:
$$ (\varphi \bot \mathrm{Id})\circ \alpha(a_0)=\varphi \circ f(a_0)+l
\mathrm{\ et\ } (\varphi \bot \mathrm{Id})\circ \alpha(b_0)=\varphi
\circ f(b_0)+m.$$
We deduce that $T \circ t_f=t_{\varphi \circ f}$.
\item
For $A_{v,w}$, we have $\alpha(a_0)=v$ and
$\alpha(b_0)=w+l'$, where $ l'$ is a non-zero element of $L'$. Consequently:
$$ (\varphi \bot \mathrm{Id})\circ \alpha(a_0)=\varphi(v)
\mathrm{\ et\ } (\varphi \bot \mathrm{Id})\circ \alpha(b_0)=\varphi
(w)+l'.$$
We have to distinguish two cases:
\begin{itemize}
\item if $\varphi(v) \in W$, since $\varphi$ preserves quadratic forms,
  we have $q(\varphi(v))=q(v)$ and, since $L'$ is orthogonal to
  $V$, $B(\varphi(v),p_W \circ
  \varphi(w))=B(\varphi(v),\varphi(w))=B(v,w).$ Thus the morphism $A_{\varphi(v), p_W \circ
    \varphi(w)}$ is defined and we have:
$T \circ A_{v,w}=A_{\varphi(v), p_W \circ \varphi(w)};$
\item otherwise, $\varphi(v)=p_{W} \circ \varphi(v)+m$ where $m$ is a
  non-zero element of $L$. Consequently, we obtain a morphism of nul
  rank and we have: $T \circ A_{v,w}=t_{p_W \circ (\varphi \bot \mathrm{Id} ) \circ \alpha}$.
\end{itemize}

The cases $B_{v,w}$ and $C_{v,w}$ are similar to the case of
$A_{v,w}$ and are left to the reader.

\item
For the morphism $D_f$, where $f$ is an element of
$\mathrm{Hom}_{\Eq}(H_0, V)$, we have $\alpha(a_0)=f(a_0)=v$ and
$\alpha(b_0)=f(b_0)=w$ where $v$ and $w$ are elements of $V$. Consequently:
$ (\varphi \bot \mathrm{Id})\circ \alpha(a_0)=\varphi(v)
\mathrm{\ and\ } (\varphi \bot \mathrm{Id})\circ \alpha(b_0)=\varphi(w).$
Since $\varphi \circ f$ preserves the quadratic forms, we have: $q(\varphi
\circ f(a_0))=q(\varphi \circ f(b_0))=0$ and $B(\varphi \circ
f(a_0),\varphi \circ f(b_0))=1$. Thus the morphisms $A_{\varphi
  \circ f(a_0), \varphi \circ f(b_0)}$, $B_{\varphi \circ f(a_0),
  \varphi \circ f(b_0)} $ and $C_{\varphi \circ f(a_0), \varphi \circ f(b_0)}$ are defined.

We have to distinguish four cases:
\begin{itemize}
\item if $\varphi(v) \in W$ and $\varphi(w) \in W$ then $T \circ D_f= D_{\varphi \circ f} $;
\item if $\varphi \circ f(a_0) \in W$  and $ \varphi \circ f(b_0)
  \notin W$, we have $\varphi \circ f(a_0)=w'$ and $ \varphi \circ
  f(b_0)=w''+l$, where $l$ is a non-zero element of $L$. Consequently,
  we obtain a morphism of type $A$ and we have $T
  \circ D_f=A_{\varphi \circ f(a_0), \varphi \circ f(b_0)}$;
\item if $\varphi \circ f(a_0) \notin W$ and $\varphi \circ f(b_0) \in
  W$, we have $\varphi \circ f(a_0)=w'+l$ and $ \varphi \circ
  f(b_0)=w''$, where $l$ is a non-zero element of $L$. Consequently,
  we obtain a morphism of type $B$ and we have $T
  \circ D_f=B_{\varphi \circ f(a_0), \varphi \circ f(b_0)}$;
\item if $\varphi \circ f(a_0) \notin W$, $\varphi \circ f(b_0) \notin
  W$ and $\varphi
\circ f(a_0+b_0) \in W$, we have $\varphi \circ f(a_0)=w'+l$ and $ \varphi \circ
  f(b_0)=w''+l$, where $l$ is a non-zero element of $L$. Consequently,
  we obtain a morphism of type $C$ and we have $T
  \circ D_f=C_{\varphi \circ f(a_0), \varphi \circ f(b_0)}$;
\item if $\varphi \circ f(a_0) \notin W$, $\varphi \circ f(b_0) \notin
  W $ and $\varphi
\circ f(a_0+b_0) \notin W$, we have $\varphi \circ f(a_0)=w'+l$ and $ \varphi \circ
  f(b_0)=w''+l'$, where $l$ and $l'$ are non-zero elements of
  $L$. Consequently, we obtain a morphism of nul rank and we have $T
  \circ D_f = t_{p_W \circ (\varphi \bot \mathrm{Id} ) \circ \alpha}$.
\end{itemize}

\end{enumerate}

\end{proof}

\subsubsection{Splitting of the filtration for the functor $P_{H_0}$}
In this paragraph, we prove the following result.

\begin{prop} \label{3.10}
The rank filtration splits for the functor $P_{H_0}$, namely:
$$P_{H_0}=P_{H_0}^{(0)} \oplus P_{H_0}^{(1)}/P_{H_0}^{(0)} \oplus
P_{H_0}/P_{H_0}^{(1)}.$$
\end{prop}
\begin{proof}
By Theorem \ref{2.6}, we have: $P_{H_0}^{(1)}=P_{H_0}^{(0)} \oplus P_{H_0}^{(1)}/P_{H_0}^{(0)}.$
To prove the proposition, it is sufficient to prove that $P_{H_0}=P_{H_0}^{(1)}
\oplus P_{H_0}/P_{H_0}^{(1)}$. 

By definition of the filtration, we have the short exact sequence:
\begin{equation}
 0 \rightarrow P_{H_0}^{(1)} \rightarrow P_{H_0} \xrightarrow{p}
P_{H_0}/P_{H_0}^{(1)} \rightarrow 0. \label{PH0-sec}
\end{equation}

Let $V$ be an object in $\Tq$, we consider a morphism $f$ of
$\mathrm{Hom}_{\Eq}(H_0, V)$ and the generator $[D_f]$ of
$P_{H_0}/P_{H_0}^{(1)}(V)$ associated to $f$. Since the map $f$
preserves the quadratic forms, we have: $q(f(a_0))=q(f(b_0))=0$, 
$q(f(a_0+b_0))=1$; thus $B(f(a_0),f(b_0))=1$. Consequently, the morphisms $A_{f(a_0), f(b_0)}$,  $B_{f(a_0), f(b_0)}$ and $C_{f(a_0),
  f(b_0)}$ of $\mathrm{Hom}_{\Tq}(H_0,V)$ are defined. We define a map $s_V: P_{H_0}/P_{H_0}^{(1)}(V) \rightarrow P_{H_0}(V)$ by:
$$ \begin{array}{cccl}
s_V : &  P_{H_0}/P_{H_0}^{(1)}(V) & \longrightarrow & P_{H_0}(V)  \\
       & [D_f] & \mapsto & [D_f]+[A_{f(a_0), f(b_0)}]+[B_{f(a_0), f(b_0)}]+[C_{f(a_0), f(b_0)}].
\end{array}$$
We verify the two following statements.
\begin{enumerate}
\item{$p_V \circ s_V=\mathrm{Id}.$\\}
For $[D_f]$ a canonical generator of $P_{H_0}/P_{H_0}^{(1)}(V)$, we have
$$p_V \circ s_V ([D_f])=[D_f]$$
since the morphisms $A_{f(a_0), f(b_0)}$, $B_{f(a_0),
  f(b_0)}$ and $C_{f(a_0), f(b_0)}$ have a rank equal to one.

\item{The maps $s_V$ define a natural map.\\}
One verifies that, for a morphism $T=V \xrightarrow{\varphi} W \bot L \xleftarrow{i_W}
W$ of $\mathrm{Hom}_{\Tq}(V,W)$, we have the commutativity of the
following diagram:
$$\xymatrix{
P_{H_0}/P_{H_0}^{(1)}(V) \ar[r]^-{s_V}
\ar[d]_{P_{H_0}/P_{H_0}^{(1)}(T)}&   P_{H_0}(V) \ar[d]^{P_{H_0}(T)}       \\
P_{H_0}/P_{H_0}^{(1)}(W) \ar[r]^-{s_W} &  P_{H_0}(W).
}$$

In order to simplify notation, we will write:
 $A'= A_{\varphi \circ f(a_0), p_W \circ \varphi \circ
  f(b_0)} $, $B'=  B_{p_W \circ \varphi \circ f(a_0), \varphi \circ
  f(b_0)}$, $C'=C_{p_W \circ \varphi \circ f(a_0), p_W \circ \varphi \circ
  f(b_0)}$ and $t'=t_{p_W \circ (\varphi \bot \mathrm{Id}) \circ
  \alpha}$.

On the one hand, by Lemma \ref{3.9}, we have:
$$\begin{array}{l}
P_{H_0}(T) \circ s_V ([D_f]) \\

 = P_{H_0}(T)([D_f]+[A_{f(a_0), f(b_0)}]+[B_{f(a_0), f(b_0)}]+[C_{f(a_0), f(b_0)}])\\

    = [T \circ D_f]+[ T \circ
  A_{f(a_0), f(b_0)}]+ [T \circ B_{f(a_0), f(b_0)}]+ [T \circ C_{f(a_0),
    f(b_0)}] 
    \end{array}$$
$$\begin{array}{l}
    = 
\left\lbrace
\begin{array}{llllll}
[D_{\varphi \circ f}] &+ [A'] &+ [B'] &+ [C'] & &\mathrm{if\ } \varphi \circ f(a_0) \in W
\mathrm{\ and\ } \varphi \circ f(b_0) \in W\\
 \lbrack A' \rbrack &+ \lbrack A'\rbrack &+ \lbrack t' \rbrack
  &+\lbrack t'\rbrack &=0&\mathrm{if\ } \varphi \circ f(a_0) \in W
\mathrm{\ and\ } \varphi \circ f(b_0) \notin W\\
\lbrack B' \rbrack &+ \lbrack t' \rbrack &+ \lbrack B' \rbrack &+
\lbrack t' \rbrack &=0&\mathrm{if\ } \varphi \circ f(a_0) \notin W
\mathrm{\ and\ } \varphi \circ f(b_0) \in W\\
\lbrack C' \rbrack &+ \lbrack t' \rbrack&+
 \lbrack t' \rbrack&+ \lbrack C' \rbrack &=0&\mathrm{if\ } \varphi \circ f(a_0) \notin W
\mathrm{\ and\ } \varphi \circ f(b_0) \notin W \\
   &     &   &     &  &\mathrm{\ and\ } \varphi
\circ f(a_0+b_0) \in W\\
\lbrack t' \rbrack &+ \lbrack t' \rbrack &+ \lbrack t' \rbrack &+
\lbrack t' \rbrack &=0 &\mathrm{if\ } \varphi \circ f(a_0) \notin W
\mathrm{\ and\ } \varphi \circ f(b_0) \notin W \\
   &      &   &     &  &\mathrm{\ and\ } \varphi
\circ f(a_0+b_0) \notin W.
\\
\end{array}
\right. \\
\\
= \left\lbrace
\begin{array}{llllll}
[D_{\varphi \circ f}] &+ [A'] &+ [B'] &+ [C'] & &\mathrm{if\ } \varphi \circ f(a_0) \in W
\mathrm{\ and\ } \varphi \circ f(b_0) \in W\\
0&&&&&\mathrm{otherwise}.
\\
\end{array}
\right. \\ 
\\   
\end{array}$$

On the other hand, by Lemma \ref{3.9}, we have:
$$P_{H_0}/P_{H_0}^{(1)}(T)([D_f])=\left\lbrace
\begin{array}{ll}
[D_{\varphi \circ f}]  &\mathrm{if\ } \varphi \circ f(a_0) \in W
\mathrm{\ and\ } \varphi \circ f(b_0) \in W\\
0&\mathrm{otherwise}
\\
\end{array}
\right.$$
since the morphisms $A$, $B$, $C$ and $t$ are zero in the quotient
$P_{H_0}/P_{H_0}^{(1)}(W)$. We deduce,
$$s_W \circ P_{H_0}/P_{H_0}^{(1)}(T)([D_f])=\left\lbrace
\begin{array}{ll}
[D_{\varphi \circ f}]  + [A'] + [B'] + [C'] & \mathrm{if\ } \varphi
\circ f(a_0) \in W\\
&
\mathrm{and\ } \varphi \circ f(b_0) \in W\\
0&\mathrm{otherwise}.
\\
\end{array}
\right.$$
\end{enumerate}

Consequently, the maps $s_V$ define a natural map which is a section
of $p$. This gives rise to the splitting of the exact sequence \ref{PH0-sec}.
\end{proof}

\subsubsection{Identification of the direct summands}
The aim of this paragraph is to identify the summands of the
decomposition given in Proposition \ref{3.10}. We begin by proving
that the morphisms of type $A$ (respectively of type $B$ and $C$)
define a subfunctor of $P_{H_0}^{(1)}/P_{H_0}^{(0)}$ which is a direct
summand of this functor.
\begin{lm} \label{3.11}
The functor $P_{H_0}^{(1)}/P_{H_0}^{(0)}$ admits the following
decomposition into direct summands:
$$P_{H_0}^{(1)}/P_{H_0}^{(0)} =F_A \oplus F_B \oplus F_C$$
where $F_A$, $F_B$ and $F_C$ are subfunctors of
$P_{H_0}^{(1)}/P_{H_0}^{(0)}$ generated by, respectively, the
morphisms of type $A$, $B$ and $C$.
\end{lm}
\begin{proof}
By Lemma \ref{3.7}, we have an isomorphism of vector spaces
$$P_{H_0}^{(1)}/P_{H_0}^{(0)}(V) =F_A(V) \oplus F_B(V) \oplus F_C(V),$$
for all objects $V$ in $\Tq$.

Consequently, it is sufficient to prove that $F_A$, $F_B$ and $F_C$
are subfunctors of $P_{H_0}^{(1)}/P_{H_0}^{(0)}$.

For $F_A$, we have to verify the commutativity of the diagram
$$\xymatrix{
F_A(V) \ar[r]^-{i_V}
\ar[d]_{F_A(T)}&   P_{H_0}^{(1)}/P_{H_0}^{(0)}(V) \ar[d]^{P_{H_0}^{(1)}/P_{H_0}^{(0)}(T)}       \\
F_A(W) \ar[r]^-{i_W} &  P_{H_0}^{(1)}/P_{H_0}^{(0)}(W)
}$$
where $T$ is a morphism of $\mathrm{Hom}_{\Tq}(V,W)$.
Let $[A_{v,w}]$ be a canonical generator of $F_A(V)$, we have by Lemma \ref{3.9}
$$T \circ A_{v,w}=\left\lbrace
\begin{array}{cc}
A_{\varphi(v), p_W \circ \varphi(w)}&\mathrm{if\ } \varphi(v) \in W\\
t_{p_W \circ (\varphi \bot \mathrm{Id} ) \circ \alpha} & \mathrm{otherwise.}
\end{array}
\right.$$
Consequently, 
$$P_{H_0}^{(1)}/P_{H_0}^{(0)}(T) \circ i_V([A_{v,w}])=\left\lbrace
\begin{array}{cc}
[A_{\varphi(v), p_W \circ \varphi(w)}]&\mathrm{if\ } \varphi(v) \in W\\
0 & \mathrm{otherwise, }
\end{array}
\right.$$
since the morphism $t_{p_W \circ (\varphi \bot \mathrm{Id} ) \circ
  \alpha}$, has nul rank.

We deduce that $P_{H_0}^{(1)}/P_{H_0}^{(0)}(T) \circ i_V([A_{v,w}])$ is
in the vector space $F_A(W)$; thus, $F_A$ is a subfunctor of $P_{H_0}^{(1)}/P_{H_0}^{(0)}$.

In the same way, by the use of values of $T \circ B_{v,w}$ and $T \circ C_{v,w}$
given in Lemma \ref{3.9}, we prove that the functors $F_B$
and $F_C$ are subfunctors of $P_{H_0}^{(1)}/P_{H_0}^{(0)}$.

\end{proof}
In the following lemma, we identify the functors $F_A$, $F_B$
and $F_C$ with certain mixed functors defined in \cite{Vespa2} and recalled in section \ref{1}.

\begin{lm} \label{3.12}
\begin{enumerate}
\item The functors $F_A$ and $F_B$ are isomorphic to the
  functor \M{0,1}.
\item  The functor $F_C$ is isomorphic to the functor \M{1,1}.
\end{enumerate}
\end{lm}
\begin{proof}
\begin{enumerate}
\item{The isomorphism $F_A \simeq \M{0,1}.$}

Let $[A_{v,w}]$ be a canonical generator of $F_A(V)$, we have, by definition,
$B(v,w)=1$. Consequently, the following linear map exists:
$$ \begin{array}{cccl}
\sigma^1_V: & F_A(V)& \rightarrow & \M{0,1}(V)\\
       & [A_{v,w}] & \longmapsto & [(w,v+w)].
\end{array}$$

The map $\sigma^1_V$ is an isomorphism, whose inverse is given by

$$ \begin{array}{cccl}
{(\sigma^1_V)}^{-1}: & \M{0,1}(V)& \rightarrow &  F_A(V)\\
       & [(v,w)] & \longmapsto & [A_{v+w,v}].
\end{array}$$

We have to verify that the maps $\sigma^1_V$ define a natural map;
namely, for a morphism $T=[V \xrightarrow{\varphi} W \bot L \xleftarrow{i_W}
W]$, that the following diagram is commutative:

$$\xymatrix{
F_A(V) \ar[r]^-{\sigma^1_V}
\ar[d]_{F_A(T)}&   \M{0,1}(V) \ar[d]^{\M{0,1}(T)}\\
F_A(W) \ar[r]^-{\sigma^1_W} &  \M{0,1}(W).
}$$
We have:
$$\begin{array}{ll}
\M{0,1}(T) \circ \sigma^1_V([A_{v,w}])&= \M{0,1}(T)[(w,v+w)]\\
&=\left\lbrace
\begin{array}{ll}
[(p_W \circ(\varphi(w)),p_W \circ(\varphi(v+w))]&\mathrm{if\ } \varphi
(v) \in W\\
0 & \mathrm{otherwise }
\end{array}
\right.
\end{array}$$
by the definition of the mixed functors, and
$$\begin{array}{ll}
 \sigma^1_W \circ F_A(T)([A_{v,w}])&= \sigma^1_W \left\lbrace
\begin{array}{ll}
[A_{\varphi(v), p_W \circ \varphi(w)}] &\mathrm{if\ } \varphi(v) \in W \\
0 & \mathrm{otherwise }
\end{array}
\right.\\
                                     &=\left\lbrace
\begin{array}{ll}
[(p_W \circ(\varphi(w)), \varphi (v) +p_W \circ(\varphi(w))]&\mathrm{if\ }  \varphi(v) \in W \\
0 & \mathrm{otherwise. }
\end{array}
\right.\\
\end{array}$$

When $\varphi(v) \in W$, we have:
$$[(p_W \circ(\varphi(w)), \varphi (v) +p_W
\circ(\varphi(w))]=[(p_W \circ(\varphi(w)),p_W \circ(\varphi(v+w))],$$
what proves the naturality of $\sigma^1$. 

Since the two following cases are very close to the previous one, we only
give the definition of the isomorphism of vector spaces and we leave the
reader to verify that they define natural equivalences.
\item{The isomorphism $F_B \simeq \M{1,0}.$}

$$ \begin{array}{cccl}
\sigma^2_V: & F_B(V)& \rightarrow & \M{0,1}(V)\\
       & B_{v,w} & \longmapsto & [(v,v+w)].
\end{array}$$

\item{The isomorphism $F_C \simeq \M{1,1}.$}

$$ \begin{array}{cccl}
\sigma^3_V: & F_C(V)& \rightarrow & \M{1,1}(V)\\
       & C_{v,w} & \longmapsto & [(v,w)]
\end{array}$$
\end{enumerate}
\end{proof}
We deduce the following proposition.
\begin{prop} \label{3.13}
The projective functor $P_{H_0}$ admits the following decomposition
into direct summands:
$$P_{H_0}=\iota(P^{\F}_{\epsilon(\FF^{\oplus 2})}) \oplus
(\M{0,1}^{\oplus 2} \oplus \M{1,1}) \oplus \kappa(\mathrm{iso}_{H_0})$$
where $\M{0,1}$ and $\M{1,1}$ are mixed functors and $\mathrm{iso}_{H_0}$
is an isotropic functor.
\end{prop}
 
\begin{proof}
This proposition is a straightforward consequence of Proposition
\ref{3.10}, Theorem \ref{2.6}, Proposition \ref{2.22} and Lemmas
\ref{3.11} and \ref{3.12}.
\end{proof}

\subsection{Decomposition of $P_{H_1}$} 
The study of the functor $P_{H_1}$ is analogous to that of the
functor $P_{H_0}$ given in the previous section. Consequently, for the
functor $P_{H_1}$, we give only the principal results without proofs.

\subsubsection{Explicit description of the subquotients of the filtration}

In this paragraph, we give basis of the vector spaces
$P_{H_1}^{(0)} (V)$, $P_{H_1}^{(1)}/P_{H_1}^{(0)} (V)$  and
$P_{H_1}/P_{H_1}^{(1)} (V)$ for $V$ an object in \Tq.

\begin{lm} \label{B0-H1}
A basis $\mathcal{B}_{H_1}^{(0)}$ of $P_{H_1}^{(0)} (V)$ is given by
the set:
$$\mathcal{B}_{H_1}^{(0)}=\{ t_f \  \mathrm{for\ } f \in
\mathrm{Hom}_{\E^f}(\FF^{\oplus 2}, \epsilon(V)) \}.$$
\end{lm}

\begin{lm} \label{types-H1}
Let $T=[H_1 \xrightarrow{f} V \bot L \xleftarrow{i_V} V]$ be a morphism of $\Tq$ which represents a canonical
generator of $P_{H_1}^{(1)}/P_{H_1}^{(0)} (V)$, and $\{ a_1, b_1 \}$ a
symplectic basis of $H_1$, then the map $f$ in $T$ has
one of the three following forms.
\begin{enumerate}
\item If $I=(f(a_1),0)$ the map $f: H_1 \rightarrow V \bot L $ is
  defined by:
$$f(a_1)=v \mathrm{\qquad et \qquad} f(b_1)=w+l$$
for $v$ and $w$ elements of $V$ satisfying $q(v)=1$ and $B(v,w)=1$
and $l$ a non-zero element of $L$.
\item If $I=(f(b_1),0)$ the map $f: H_1 \rightarrow V \bot L $ is
  defined by:
$$f(a_1)=v+l \mathrm{\qquad et \qquad} f(b_1)=w$$
for $v$ and $w$ elements of $V$ satisfying $q(w)=1$ and $B(v,w)=1$
and $l$ a non-zero element of $L$.
\item If $I=(f(a_1+b_1),1)$ the map $f: H_1 \rightarrow V \bot L $ is
  defined by:
$$f(a_1)=v+l \mathrm{\qquad et \qquad} f(b_1)=w+l$$
for $v$ and $w$ elements of $V$ satisfying $q(v+w)=1$ and $B(v,w)=1$
and $l$ a non-zero element of $L$.
\end{enumerate}
\end{lm}

%\begin{proof}
%L'espace $H_1$ admet trois sous-espaces de dimension un qui sont:
%$\mathrm{Vect}(a_1)$, $\mathrm{Vect}(b_1)$ et
%$\mathrm{Vect}(a_1+b_1)$ qui sont isométriques à $(x,1)$. Ces trois sous
%espaces fournissent chacune des trois applications $f$ données dans
%l'énoncé.
%\end{proof}
\begin{nota}
The morphisms $[H_1 \xrightarrow{f} V \bot L \xleftarrow{i_V} V]$, where
$f$ is one of the morphisms described in the point $(1)$ (respectively
$(2)$ and $(3)$) of the previous lemma will be known as type $E$
(respectively $F$ and $G$) morphisms.
\end{nota}
The analogous proposition to Proposition \ref{3.4} holds for $H_1$. This justifies the following notation.
\begin{nota}
Denote by $E_{v,w}$, $F_{v,w}$ and $G_{v,w}$ the morphisms of
$\mathrm{Hom}_{\Tq}(H_1, V)$ respectively of type $E$, $F$ and $G$ and
such that $p_V \circ f(a_1)=v$ and $p_V \circ f(b_1)=w$.
\end{nota}
We deduce the following lemmas.
\begin{lm} \label{B1-PH1}
A basis
$\mathcal{B}_{H_1}^{(1)}$ of $P_{H_1}^{(1)}/P_{H_1}^{(0)} (V)$ is
given by the set:
$$\begin{array}{ll}
\mathcal{B}_{H_1}^{(1)}=\{ & [E_{v,w}] \mathrm{\ for\ } v \mathrm{\ and\
  } w \mathrm{\ elements\ of\ } V \mathrm{\ satisfying\ } q(v)=1 \mathrm{\ and\ } B(v,w)=1, \\
                        & [F_{v,w}] \mathrm{\ for\ } v \mathrm{\ and\
  } w \mathrm{\ elements\ of\ } V \mathrm{\
    satisfying\ } q(w)=1 \mathrm{\ and\ } B(v,w)=1, \\
                         &  [G_{v,w}] \mathrm{\ for\ } v \mathrm{\ and\
  } w \mathrm{\ elements\ of\ } V \mathrm{\ satisfying\ } q(v+w)=1 \mathrm{\ and\ } B(v,w)=1 \}
\end{array}$$ 
\end{lm}

\begin{lm} \label{B2-bis}
A basis $\mathcal{B}_{H_1}^{(2)}$ of $P_{H_1}/P_{H_1}^{(1)} (V)$ is
given by the set:
$$\mathcal{B}_{H_1}^{(2)}=\{ [H_f] \  \mathrm{for\ } f \in
\mathrm{Hom}_{\Eq}(H_1, V) \},$$ 
where $H_f$ is the morphism of $\Tq$ represented by the diagram: $H_1 \xrightarrow{f} V
\xleftarrow{\mathrm{Id}} V$.
\end{lm}

The rules of composition of morphisms $E_{v,w}$, $F_{v,w}$,
$G_{v,w}$ and $H_f$ are similar to those given for $A_{v,w}$,
$B_{v,w}$, $C_{v,w}$ and $D_f$ in Lemma \ref{3.9}. The details can be provided by the reader.

\subsubsection{Splitting of the filtration for the functor  $P_{H_1}$}

\begin{prop} \label{scindement-PH1}
The rank filtration splits for the functor $P_{H_1}$, namely:
$$P_{H_1}=P_{H_1}^{(0)} \oplus P_{H_1}^{(1)}/P_{H_1}^{(0)} \oplus
P_{H_1}/P_{H_1}^{(1)}.$$
\end{prop}
\begin{proof}
One verifies that the map $s_V: P_{H_1}/P_{H_1}^{(1)}(V) \rightarrow
P_{H_1}(V)$ given by:
$$ \begin{array}{cccl}
s_V : &  P_{H_1}/P_{H_1}^{(1)}(V) & \longrightarrow & P_{H_1}(V)  \\
       & [H_f] & \mapsto & [H_f]+[E_{f(a_1), f(b_1)}]+[F_{f(a_1), f(b_1)}]+[G_{f(a_1), f(b_1)}].
\end{array}$$
defines a natural map which is a section of the projection $P_{H_1} \rightarrow  P_{H_1}/P_{H_1}^{(1)}$.
\end{proof}

\subsubsection{Identification  of the direct summands}
We have the following lemma.
\begin{lm} \label{PH1-1}
The functor $P_{H_1}^{(1)}/P_{H_1}^{(0)}$ admits the following
decomposition into direct summands
$$P_{H_1}^{(1)}/P_{H_1}^{(0)} =F_E \oplus F_F \oplus F_G$$
where $F_E$, $F_F$ and $F_G$ are subfunctors of
$P_{H_1}^{(1)}/P_{H_1}^{(0)}$ generated by, respectively, the
morphisms of type $E$, $F$ and $G$.
\end{lm}
In the following lemma, we identify the functors $F_E$, $F_F$
and $F_G$ with mixed functor.

\begin{lm} \label{PH1-2}
The functors $F_E$, $F_F$ and $F_G$ are equivalent to the functor \M{1,1}.
\end{lm}
%\begin{proof}
%Les applications linéaires fournissant les équivalences naturelles
%sont construites de la même manière que celles données pour le
%foncteur $P_{H_0}$ à la proposition \ref{PH0-2}.
%\end{proof}
We deduce the following decomposition.

\begin{prop} \label{decompo-H1} 
The projective functor $P_{H_1}$ admits the following decomposition
into direct summands:
$$P_{H_1}=\iota(P^{\F}_{\epsilon(\FF^{\oplus 2})}) \oplus
\M{1,1}^{\oplus 3}  \oplus \kappa(\mathrm{iso}_{H_1})$$
where $\M{1,1}$ is a mixed functor and
$iso_{H_1}$ is an isotropic functor.
\end{prop}

\subsection{Consequences of decompositions of functors $P_{H_0}$ and
  $P_{H_1}$} \label{consequences}

In this section, we draw the conclusions of the decompositions of
$P_{H_0}$ and $P_{H_1}$ given in Propositions \ref{3.13} and
\ref{decompo-H1}. We deduce the
indecomposability of the functors $\M{0,1}$ and $\M{1,1}$, we study
the projectivity of the first isotropic functors in $\Fq$ and we give the
classification of the ``small'' simple objects of $\Fq$.

\subsubsection{Indecomposability of functors \M{0,1} and \M{1,1}}
The aim of this paragraph is to prove the following result:
\begin{prop} \label{3.24}
The functors \M{0,1} and \M{1,1} are indecomposable.
\end{prop}

The proof of this proposition relies on the following obvious lemma.

\begin{lm}
If the functor $F$ of \Fq\ decomposes as a direct sum: $F_1 \oplus \ldots \oplus
F_n$, then the projections $\pi_i: F \rightarrow F_i$ and the
inclusions $j_i: F_i \rightarrow F$ induce idempotents $e_i=j_i
\circ \pi_i$ in the ring $\mathrm{End}(F)$.
\end{lm}

\begin{proof}[Proof of Proposition \ref{3.24}]
\begin{enumerate}
\item
By the Yoneda lemma, we have:
$$\mathrm{Hom}(P_{H_0},\mathrm{Mix}_{0,1})=\mathrm{Mix}_{0,1}(H_0).$$
By a calculation, we obtain that the dimension of the space
$\M{0,1}(H_0)$ is equal to $4$. According to Proposition
\ref{3.13}, the order of multiplicity of the
summand $\mathrm{Mix}_{0,1}$ in the decomposition of $P_{H_0}$ is equal to
$2$. Consequently, the dimension of the vector space
$E:=\mathrm{Hom}(\mathrm{Mix}_{0,1},\mathrm{Mix}_{0,1})$ is $2$. We
have the following basis: $\{Id, \tau \}$ where the map $\tau$ is
given by: $\tau (\lbrack (u,v) \rbrack) = (\lbrack (v,u) \rbrack)$.
Consequently, $E=( \{ 0,Id,\tau,Id+\tau
\},+,\circ )$, as a ring, and it is easy to see that this ring does not
admit a non-trivial idempotent.
\item
Similarly, we have
$\mathrm{Hom}(P_{H_0},\mathrm{Mix}_{1,1})=\mathrm{Mix}_{1,1}(H_0)$
and\\ $\mathrm{dim}(\M{1,1}(H_0))=2$. The order of multiplicity of the
summand $\mathrm{Mix}_{1,1}$ in the decomposition of $P_{H_0}$ is equal to
$1$. We deduce that the ring
$\mathrm{Hom}(\mathrm{Mix}_{1,1},\mathrm{Mix}_{1,1})$ does not
admit a non-trivial idempotent.
\end{enumerate}
\end{proof}

We deduce from this proposition the following result, which complements
Theorem \ref{1.9}, obtained in \cite{Vespa2}:
\begin{cor}
The short exact sequence
$0 \rightarrow \m{\alpha,1} \rightarrow \M{\alpha,1} \rightarrow \m{\alpha,1} \rightarrow 0$
does not split.
\end{cor}

\subsubsection{Projectivity of certain isotropic functors in \Fq}

The decompositions given in Propositions \ref{3.13} and
\ref{decompo-H1} allow us to study the projectivity of isotropic functors
in $\Fq$. Corollary $4.37$ in \cite{math.AT/0606484} shows that the set of functors $\{ iso_V | V \in \mathcal{S} \}$ is a set of
projective generators of $\Gq$, where $\mathcal{S}$ is a set of
representatives of isometry classes of (possibly degenerate) quadratic spaces.

Since the functor $\kappa(iso_{H_0})$ (respectively
$\kappa(iso_{H_1})$ ) is a direct summand of the functor $P_{H_0}$
(respectively $P_{H_1}$) we have the following result.
\begin{prop} \label{3.27}
The functors $\kappa(iso_{H_0})$ and $\kappa(iso_{H_1})$ are
projective in the category \Fq.
\end{prop}

We deduce from Corollary \ref{1.6} and the previous proposition, the
following result.

\begin{cor}
The category $\Fq$ contains non-constant finite, projective objects.
\end{cor}

This corollary constitutes one of the new features of the category
$\Fq$ compared to $\F$. Recall that, according to Corollary B7 in
\cite{KuhnI}, due to Lionel Schwartz, the category $\F$ does not contain
non-constant finite projective functors.

Recall that, the functor $\kappa(iso_{(x,0)})$ is the top composition factor of \M{0,1}\ and $\kappa(iso_{(x,1)})$ is that of
  \M{1,1}. 
We have the following result.
\begin{prop} \label{3.29}
The projective cover of $\kappa(iso_{(x,0)})$ (respectively
  $\kappa(iso_{(x,1)})$) is the functor $\M{0,1}$ (respectively
  $\M{1,1}$). In particular, the functors $\kappa(iso_{(x,0)})$ and $\kappa(iso_{(x,1)})$ are not
projective in \Fq. 
\end{prop}
\begin{proof}
Since $\kappa(iso_{(x,0)})(H_0) \ne \{ 0\}$ and $\kappa(iso_{(x,1)})(H_0) \ne \{ 0\}$,
 if these two functors were projective, they would be direct summands
 of the functor $P_{H_0}$. We deduce from Proposition \ref{3.13}, that
 these functors are not projective. \end{proof}
\begin{rem}
Propositions \ref{3.27} and \ref{3.29} let us conjecture that, for a
nondegenerate $\FF$-quadratic space $H$, $\kappa(iso_H)$ is a
projective functor in $\Fq$ and, for a degenerate quadratic space $D$,
$\kappa(iso_D)$ is not a projective functor in $\Fq$ and its
projective cover is a generalized mixed functor. This result will be
the subject of future work.
\end{rem}

\subsubsection{Classification of simple objects $S$ of  $\Fq$ such
  that either $S(H_0) \ne 0$ or  $S(H_1) \ne 0$}
If $S$ is a simple object in $\Fq$, such that $S(H_0) \ne 0$, the Yoneda lemma implies that $\mathrm{Hom}(P_{H_0},S)=S(H_0) \ne
0.$ Consequently, there exists a morphism of $\Fq$ from $P_{H_0}$
to $S$ which is an epimorphism, by simplicity of $S$. We deduce from
the decompositions given in Proposition \ref{3.13} and \ref{decompo-H1},
from Corollary \ref{1.7} concerning the functors $iso_{H_0}$ and
$iso_{H_1}$ and from the study of the functors $\M{0,1}$ and $\M{1,1}$
done in \cite{Vespa2} and recalled in section \ref{1}, the following result.

\begin{prop} \label{3.30}
The isomorphism classes of non-constant simple functors of $\Fq$ such
that either  $S(H_0) \ne 0$ or  $S(H_1) \ne 0$ are:
$$\iota(\Lambda^1),\ \iota(\Lambda^2),\ \iota(S_{(2,1)}),\
\kappa(iso_{(x,0)}),\ \kappa(iso_{(x,1)}),\ R_{H_0},\ R_{H_1},\ S_{H_1}$$
where $ R_{H_0},\ R_{H_1}$ and $ S_{H_1}$ are the simple functors
introduced in Corollary \ref{1.7}.
\end{prop}

\subsubsection{Extension groups in $\Fq$}
By Theorem \ref{1.4} and \ref{1.5} we obtain an exact, fully-faithful functor
$\prod_{V \in \mathcal{S}} \FF[O(V)]-mod \xrightarrow{\tilde{\kappa}} \Fq,$
where $\mathcal{S}$ is a set of representatives of isometry classes of
quadratic spaces (possibly degenerate). Consequently, for $M$ and $N$ two
$\FF[O(V)]-$ modules, this functor induces a morphism of extension groups:
$$\mathrm{Ext}^*_{\FF[O(V)]-mod}(M,N) \xrightarrow{(\tilde{\kappa})_*}
\mathrm{Ext}^*_{\Fq}(\tilde{\kappa}(M),\tilde{\kappa}(N)).$$
We have the following proposition.
\begin{prop} \label{3.31}
For $V \in \{ H_0, H_1 \}$, the morphism $(\tilde{\kappa})_*$ is an isomorphism.
\end{prop}
The proof of this proposition relies on the following lemma.
\begin{lm} \label{3.32}
For $V \in \{ H_0, H_1 \}$, if $P$ is a finite projective $\FF[O(V)]$-module,
$\tilde{\kappa}(P)$ is projective in $\Fq$.
\end{lm}
\begin{proof}
If $P$ is a finite projective $\FF[O(V)]$-module, there exists a $\FF[O(V)]$-module $Q$ such that $P \oplus Q \simeq \FF[O(V)]^{\oplus N}.$
We deduce from the exactness of $\kappa$ that $\tilde{\kappa}(P \oplus Q) \simeq
\tilde{\kappa}(P) \oplus \tilde{\kappa}(Q)$. Since
$\tilde{\kappa}(\FF[O(V)])=\kappa(iso_V)$ and the functors $\kappa(iso_{H_0})$ and $\kappa(iso_{H_1})$ are projective, by
Proposition \ref{3.27}, we obtain that $\tilde{\kappa}(P)$ is projective.
\end{proof}

\begin{proof}[Proof of Proposition \ref{3.31}]
Let $M$ and $N$ be $\FF[O(V)]$-modules for $V \in \{ H_0, H_1 \}$ and
$P_\bullet \rightarrow M$ be a projective resolution of $M$. Lemma
\ref{3.32} implies that $\tilde{\kappa}(P_\bullet)$ is a projective
resolution of $\tilde{\kappa}(M)$. The functor $\tilde{\kappa}$
induces a morphism of cochain complexes
$$\mathrm{Hom}_{\FF[O(V)]-mod}(P_\bullet,N) \rightarrow
\mathrm{Hom}_{\Fq}(\tilde{\kappa}(P_\bullet),\tilde{\kappa}(N)) $$
which induces the morphism $(\tilde{\kappa})_*$ in cohomology. Since the functor $\tilde{\kappa}$ is fully-faithful the previous morphism is an isomorphism and so induces an isomorphism in cohomology.

\end{proof}

We deduce the following corollary:
\begin{cor} \label{3.33}
For $n$ a natural number, we have:
$$\mathrm{Ext}^n_{\Fq}(R_{H_0},R_{H_0}) \simeq \FF \mathrm{\quad and
  \quad} \mathrm{Ext}^n_{\Fq}(R_{H_1},R_{H_1}) \simeq \FF $$
where $R_{H_0}$ and $R_{H_1}$ are the simple functors introduced in
Corollary \ref{1.7}.
\end{cor}
\begin{proof}
Let $\epsilon$ be an element in $\{0,1\}$. For $V=H_{\epsilon}$, by Corollary \ref{1.7} (1), we have $\tilde{\kappa}(\FF)=R_{H_{\epsilon}}$. So, applying Proposition \ref{3.31} to $M=N=\FF$ we obtain:
$$\mathrm{Ext}^*_{\Fq}(R_{H_{\epsilon}},R_{H_{\epsilon}}) \simeq \mathrm{Ext}^*_{\FF[O(H_{\epsilon})]-mod}(\FF,\FF)=H^*(O(H_{\epsilon}),\FF).$$
Since $O(H_0) \simeq \mathfrak{S}_2 \simeq C_2$ and $O(H_1)\simeq \mathfrak{S}_3 \simeq GL_{2}(\FF)$ we know by classical results of cohomology of groups that 
$$H^n(O(H_0) ,\FF)=H^n(O(H_1) ,\FF)=\FF.$$
\end{proof}

\begin{rem}
This corollary exhibits an important difference between the
categories $\F$ and $\Fq$; recall that in $\F$
 $\mathrm{Ext}^1_{\F}(S,S)=0$ for all simple objects $S$ of $\F$ (see \cite{piriou-schwartz-K-theorie}). 
\end{rem}

%%%%%%%%%%%%%%%%%%%%%%%%%%%%%%%%%%%%%%%%%
\section{Application: the polynomial functors of $\Fq$} \label{4}
In this section, having generalized the notion of polynomial
functor to the category \Fq, we prove, by induction, that the polynomial functors of $\Fq$ are in the image of the functor $\iota:
\F \rightarrow \Fq$.

\subsection{Definition of polynomial functors of  \Fq}

\subsubsection{The difference functors of  \Fq}

We define the difference functors of $\Fq$ which generalize the notion of
difference functor of $\F$. Recall that, according to \cite{Sch}, 
the difference functor $\Delta: \F \rightarrow \F$ is the functor
given by
$$ \Delta F(V):= \mathrm{Ker} (F(V \oplus \FF)  \xrightarrow{F(p)}
F(V)),$$
for $F$ an object in $\F$, $V$ an object in $\E^f$ and $p: V \oplus \FF
\rightarrow V$ the projection.

\begin{defi}
The difference functors $\Delta_{H_0}: \Fq \rightarrow \Fq$ and
$\Delta_{H_1}: \Fq \rightarrow \Fq$ are the functors defined by:
$$ \Delta_{H_0} F(V):= \mathrm{Ker} (F(V \bot H_0)  \xrightarrow{F(T_0)}
F(V)),$$
$$ \Delta_{H_1} F(V):= \mathrm{Ker} (F(V \bot H_1)  \xrightarrow{F(T_1)}
F(V)),$$
for $F$ an object in $\Fq$, $V$ an object in $\Tq$, and $T_i=[V \bot
H_i \xrightarrow{\mathrm{Id}} V \bot H_i \xleftarrow{i_V} V]$ for $i
\in \{0,1\}$ .
\end{defi}
We have the following result:
\begin{lm} \label{4.2}
The functors $\Delta_{H_0}$ and $\Delta_{H_1}$ are exact.
\end{lm}

\subsubsection{Definition of polynomial functors}
Before giving the definition of polynomial functors in $\Fq$, let us recall
that of polynomial functors of $\F$ ( \cite{Sch}). For
an object $F$ of $\F$, $F$ is a polynomial functor of
degree $0$ if and only if $\Delta F=0$ and, for an integer $d$, $F$ is
polynomial of degree at most $d+1$ if and only if $\Delta F$ is
polynomial of degree at most $d$.

\begin{defi} \label{4.3}
Let $F$ be an object in $\Fq$:
\begin{enumerate}
\item
the functor $F$ is polynomial of degree $0$ if and only if $\Delta_{H_0}F=\Delta_{H_1}F=0$;
\item
for an integer $d$, the functor $F$ is polynomial of degree at most
$d+1$, if and only if $\Delta_{H_0}F$ and $\Delta_{H_1}F$ are
polynomial of degree at most $d$.
\end{enumerate}
\end{defi}

The following proposition allows us to simplify the definition of a
polynomial functor of degree $0$.

\begin{lm} \label{4.4}
Let $F$ be an object in \Fq. The functor $\Delta_{H_0}F$ is zero if
and only if the functor $\Delta_{H_1}F$ is zero. 
\end{lm}
\begin{proof}
If $\Delta_{H_0}F=0$, we have, for all objects $V$ of $\Tq$, $F(V) \simeq F(V \bot H_0).$
Consequently,
$$F(V) \simeq F(V \bot H_0) \simeq F(V \bot H_0 \bot H_0) \simeq F(V
\bot H_1 \bot H_1),$$
where the last isomorphism is obtained from the isomorphism $ H_0 \bot H_0 \simeq  H_1 \bot H_1,$ recalled in section \ref{1}.
We deduce from the existence of morphisms $F(V) \hookrightarrow F(V \bot
H_1)$ and $F(V \bot H_1) \hookrightarrow F(V \bot H_1 \bot H_1)$,
induced by the inclusions and from the isomorphism between $F(V)$ and $F(V \bot H_1 \bot H_1)$, that
$$F(V) \simeq F(V \bot H_1) \simeq F(V \bot H_1 \bot H_1).$$
Therefore $\Delta_{H_1}F=0$.

The proof of the converse is similar.
\end{proof}

\subsection{Study of polynomial functors of $\Fq$}
The aim of this section is to prove the following result:
\begin{thm} \label{4.5}
The polynomial functors of $\Fq$ are in the image of the functor $\iota: \F \rightarrow \Fq$.
\end{thm}
We will prove this theorem by induction over the degree of the polynomial functors.

\subsubsection{Polynomial functors of degree zero of  $\Fq$}
In this paragraph, we start the induction. The proof of the following
result relies, in an essential way, on the classification of simple functors $S$ of
$\Fq$ such that $S(H_0) \ne 0$ or $S(H_1) \ne
0$, obtained in Proposition \ref{3.30}.

\begin{lm} \label{4.6}
Let $S$ be a simple functor of $\Fq$, $S$ is a polynomial functor of
degree zero if and only if $S$ is the constant functor $\FF$.
\end{lm}
\begin{proof}

In order to prove the direct implication, we have to distinguish two cases.
\begin{enumerate}
\item{If $S(H_0)=S(H_1)=0$.}

By the classification of the nondegenerate quadratic spaces over 
$\FF$, if $W$ is a space of minimal dimension, satisfying $S(W) \ne
0$, we have the existence of an element $\epsilon$ of $\{0,1 \}$ and a
nondegenerate quadratic space $V$ which is non-zero, such that: $W \simeq H_{\epsilon} \bot V.$
Since $W$ is of minimal dimension, we have $S(V)=0$. This implies: 
$$\Delta_{H_{\epsilon}}S(V)=S(H_{\epsilon} \bot V) \ne 0.$$
We deduce the result in this case.
\item{If $S(H_0) \ne 0$ or $S(H_1) \ne 0$.}

In this case, we use the classification of simple functors $S$ of
$\Fq$ such that $S(H_0) \ne 0$ or $S(H_1) \ne
0$ obtained in Proposition \ref{3.30}. By an explicit calculation for all
the functors $S$ obtained in this classification, we obtain that the
functors $\Delta_{H_0}S$ are non-zero except for the constant functor $S=\FF$. 

\end{enumerate}
The converse is trivial.

\end{proof}

\subsubsection{Proof of Theorem \ref{4.5} }

To prove Theorem \ref{4.5}, we need the following result where the 
idempotents $[e_V]$, obtained in Proposition \ref{2.15}, play a
crucial r\^ole.

\begin{prop} \label{4.7}
Let $S$ be a non-trivial simple functor of $\Fq$ which is not in the
image of the functor $\iota: \F \rightarrow \Fq$, then, one of the
functors $\Delta_{H_0}S$ or $\Delta_{H_1}S$ is not in the image of the
functor $\iota: \F \rightarrow \Fq$.
\end{prop}
\begin{proof}
Let $W$ be a nondegenerate quadratic space of minimal dimension, such
that $S(W) \ne 0$. We distinguish the two following cases.
\begin{enumerate}
\item{If $\mathrm{dim}(W)=2$.}

By an explicit calculation for all
the functors $S$ of the classification given in Proposition
\ref{3.30}, we obtain the result.
\item{If $\mathrm{dim}(W)>2$.}

There exists a nondegenerate quadratic space $V$, possibly trivial,
and an element $\epsilon$ of $\{0,1 \}$, such that:
$W \simeq H_0 \bot H_{\epsilon} \bot V.$
Suppose that $\Delta_{H_0}S$ and $\Delta_{H_1}S$ are in the image of
the functor $\iota$, we prove, below, that this implies
that $S$ is in the image of $\iota$. By Lemma
\ref{2.21} it is sufficient to show the existence of an object $W$ in $\Tq$ such that:
$S(e_W)S(W) \ne 0.$ By Lemma \ref{2.17}, we have: $e_W=e_{H_0}
\bot e_{H_{\epsilon}} \bot e_V.$ 

Since $W$ is assumed to be a space of minimal dimension such that $S(W) \ne 0$, we
have
 $S(H_0 \bot V)=S(H_{\epsilon} \bot V)=0.$ This implies that
\begin{eqnarray}
\Delta_{H_0}S(H_{\epsilon} \bot V) \simeq S(W) \label{poly-eqn1}\\
\Delta_{H_{\epsilon}}S(H_{0} \bot V) \simeq S(W).\label{poly-eqn2}
\end{eqnarray}
These isomorphisms are natural and, for (\ref{poly-eqn1}), the action of $\mathrm{End}_{\Tq}(H_{\epsilon} \bot V)$ on
$\Delta_{H_0}S(H_{\epsilon} \bot V)$ corresponds to the restriction of
the action of $\mathrm{End}_{\Tq}(W)$ on $S(W)$. In the same way, for
(\ref{poly-eqn2}), the action of $\mathrm{End}_{\Tq}(H_{0} \bot V)$ on
$\Delta_{H_{\epsilon}}S(H_0 \bot V)$ corresponds to the restriction of
the action of $\mathrm{End}_{\Tq}(W)$ on $S(W)$. Suppose that $\Delta_{H_0}S$ and
$\Delta_{H_1}S$ are in the image of $\iota$. We deduce that:
$$S(1_{H_{0}} \bot e_{H_{\epsilon}} \bot  e_V)S(W) = \Delta_{H_{0}}S(e_{H_{\epsilon}}
\bot e_V) \Delta_{H_0}S(H_{\epsilon} \bot V)$$
$$= \Delta_{H_0}S(H_{\epsilon} \bot V) = S(W)$$
where the first equality comes from the action described
previously, the second is a consequence of Lemma \ref{2.17} and
the third is given by \ref{poly-eqn1}. In the same way, we obtain:
$$S(e_{H_{0}} \bot 1_{H_{\epsilon}} \bot  e_V)S(W) =\Delta_{H_{\epsilon}}S(e_{H_{0}}
\bot e_V) \Delta_{H_{\epsilon}}S(H_{0} \bot V)$$
$$=\Delta_{H_{\epsilon}}S(H_{0} \bot V)=S(W).$$
We deduce that: $S(1_{H_{0}} \bot e_{H_{\epsilon}} \bot  e_V) \circ S(e_{H_{0}} \bot
1_{H_{\epsilon}} \bot  e_V) S(W)=S(W).$
Since $S(1_{H_{0}} \bot e_{H_{\epsilon}} \bot  e_V) \circ S(e_{H_{0}} \bot
1_{H_{\epsilon}} \bot  e_V)=S(e_W)$ by Lemma \ref{2.17}, we have:
$S(e_W)S(W) \ne 0$, as required.
\end{enumerate}

\end{proof}
The proof of the theorem relies, also, on the following lemmas.
\begin{lm} \label{4.8}
\begin{enumerate}
\item
A functor $F$ of $\Fq$ which takes values in finite vector spaces and
such that  $\Delta_{H_{\epsilon}} F$ is finite, is finite.
\item
A polynomial functor $F$ of $\Fq$ which takes values in finite vector spaces is finite.
\end{enumerate}
\end{lm}
\begin{proof} 
\begin{enumerate}
\item
The functor 
$$\begin{array}{ccc}
\Fq & \xrightarrow{\gamma} & \E \times \Fq \\
F    &\mapsto & (F(0), \Delta_{H_{\epsilon}}F)
\end{array}$$
is exact and faithful by Lemma \ref{4.2}. Hence, if $\gamma(F)$ is finite then $F$ is finite.
\item
If $F$ takes values in finite vector spaces, so does $\Delta_{H_{\epsilon}}F$. Consequently we can apply the first point recursively to obtain the result.
\end{enumerate}
\end{proof}

\begin{lm} \label{4.9}
A finite object $F$ of $\Fq$ whose composition factors are in the
image of the functor $\iota$ is in the image of the functor $\iota$.
\end{lm}
\begin{proof}
This result is a straightforward consequence of the thickness of the
subcategory $\iota(\F)$ in $\Fq$ given in Theorem \ref{2.19}.
\end{proof}

\begin{proof}[Proof of Theorem \ref{4.5}]
Since a functor of $\Fq$ is colimit of its subfunctors which take values in finite vector spaces, we deduce from Lemma \ref{4.8}, that it is sufficient to prove the result for the finite polynomial functors of $\Fq$. Furthermore, since the functors 
$\Delta_{H_0}$ and $\Delta_{H_1}$ are exact by Lemma \ref{4.2},
it is sufficient to consider the case of a simple functor $S$. We
prove the theorem by induction over the polynomial degree.

If $S$ is polynomial of degree $0$, according to Lemma \ref{4.2},
$S$ is in the image of the functor $\iota$.

Suppose that all simple polynomial functors of $\Fq$ and of degree $d$
are in the image of the functor $\iota$ and consider a simple
polynomial functor $S$ of $\Fq$ such that $\mathrm{deg}(S)=d+1$. By
the definition of polynomial functor in $\Fq$ given in \ref{4.3}, the
functors $\Delta_{H_0}S$ and $\Delta_{H_1}S$ are polynomial of degree
$d$. We deduce that all composition factors of $\Delta_{H_0}S$ and
$\Delta_{H_1}S$ are polynomial of degree smaller than or equal to $d$ and,
by induction, we obtain that they are in the image of $\iota$. Since
$S$ is a simple functor, it is a quotient of a standard projective
functor $P_V$. Consequently $S$ takes its values in finite dimensional
vector spaces. Therefore, $\Delta_{H_0}S$ and
$\Delta_{H_1}S$ take their  values in finite dimensional
vector spaces. We deduce from Lemma \ref{4.8} that the functors
$\Delta_{H_0}S$ and $\Delta_{H_1}S$ are finite, and, by Lemma
\ref{4.9}, we obtain that $\Delta_{H_0}S$ and $\Delta_{H_1}S$ are in
the image of the functor $\iota$. Consequently, by Proposition
\ref{4.7}, $S$ is in the image of the functor $\iota$.
\end{proof}

%%%%%%%%%%%%%%%%%%%%%%%%%%%%%%%%%%%%%%%%%%%%%%%%%%%%%%%%%%%%%%%%%%%%
\bibliographystyle{amsplain}
\bibliography{these}

\end{document}